\documentclass[12pt]{article}

\usepackage{latexsym}
\usepackage{amsfonts}
\usepackage{amssymb}
\usepackage{amsmath}
\usepackage{amsthm}
\usepackage{epsfig}
\usepackage[english]{babel}
\usepackage[utf8]{inputenc}
\usepackage{graphicx}
\usepackage[colorinlistoftodos]{todonotes}
\usepackage{mathrsfs}
\usepackage{enumerate}
\usepackage{amscd}
\usepackage{tikz-cd}
\usepackage{hyperref}

\newtheorem{theorem}{Theorem}[section]
\newtheorem{corollary}[theorem]{Corollary}
\newtheorem{lemma}[theorem]{Lemma}
\newtheorem{proposition}[theorem]{Proposition}

\theoremstyle{definition}
\newtheorem{definition}[theorem]{Definition}

\theoremstyle{remark}
\newtheorem{remark}[theorem]{Remark}
\newtheorem{example}[theorem]{Example}

\numberwithin{equation}{section}
\newcommand\numberthis{\addtocounter{equation}{1}\tag{\theequation}}

\newcommand{\ms}{\mathscr}
\newcommand{\mb}{\mathbb}
\newcommand{\mr}{\mathrm}

\newcommand{\mc}{\mathcal}

\newcommand{\norm}[1]{\left\lvert#1\right\rvert}
\newcommand{\End}{\mathrm{End}}

\newcommand{\PA}{\mc{P}_{\mc{A}^{0,\bullet}(X)}}
\newcommand{\Db}{\mc{A}^{0,\bullet}(X)}
\newcommand{\del}{\partial}
\newcommand{\delbar}{\bar{\partial}}
\newcommand{\pr}{\prime}
\newcommand{\dpr}{\prime\prime}
\newcommand{\tr}{\mr{Tr}_s}
\newcommand{\tE}{\widetilde{E}^{\bullet}}
\newcommand{\mE}{\mb{E}}
\newcommand{\tmE}{\widetilde{\mE}}
\newcommand{\mEh}{{\mE}_h}
\newcommand{\mEhp}{{\mE}^{\pr}_h}
\newcommand{\mEdp}{{\mE}^{\dpr}}
\newcommand{\tR}{\widetilde{\mc{R}}}
\newcommand{\Rh}{\mc{R}_h}
\newcommand{\mM}{\mc{M}}

\begin{document}
	\title{On the Bott-Chern characteristic classes for coherent sheaves} 
	\author{Hua Qiang}
	\date{}
	\maketitle
	
	\begin{abstract}
		In the paper \cite{Block2010}, Block constructed a dg-category $\mc{P}_{\mc{A}^{0, \bullet}(X)}$ using cohesive modules which is a dg-enhancement of $D^b_{\mr{Coh}}(X)$, the bounded derived category of complexes of analytic sheaves with coherent cohomology. In this article, we construct natural superconnections on cohesive modules and use them to define characteristic classes with values in Bott-Chern cohomology. In addition, we generalize the double transgression formulas in \cite{Bismut1988} \cite{Botta} \cite{Donaldson1987} and prove the invariance of these characteristic classes under derived equivalences. This provides an extension of Bott-Chern characteristic classes to coherent sheaves on complex manifolds.
	\end{abstract}
	
	\tableofcontents
	
	\section{Introduction}
		Traditionally, the complex structure of a complex manifold $X$ is encoded in the sheaf of holomorphic functions $\mc{O}_X$. For applications to noncommutative geometry, such local constructions are not available and we are forced to use global differential geometric constructions. When the manifold is projective, every coherent sheaf $\ms{S}$ admits a global resolution by holomorphic vector bundles 
		$$ 0 \rightarrow E^n \rightarrow E^{n-1} \rightarrow ... \rightarrow E^1 \rightarrow E^0 \rightarrow \ms{S} \rightarrow0$$
		and we can apply the theory of holomorphic vector bundles to study $\ms{S}$. However, for a general compact complex manifold that is not projective, there may not exist a global resolution by holomorphic vector bundles and alternative methods are required. 
	
		Even though the sheaf of holomorphic sections of $\ms{S}$ does not admit global resolution by holomorphic vector bundles, the underlying sheaf of real analytic sections does admit a global resolution by real analytic vector bundles, see \cite{Atiyah1961}. In \cite{Block2010}\cite{Block2006}, Block constructed higher order differentials on such resolutions and the resulting geometric objects are called cohesive modules. He further showed that they form a dg-category $\mc{P}_{A^{0,\bullet}(X)}$ whose homotopy category is equivalent to $D^b_{\mr{Coh}}(X)$, the bounded derived category of complexes of sheaves of $\mc{O}_X$-modules with coherent cohomology. Using the theory of superconnections defined in \cite{Quillen1985}, it's natural to extend the classical constructions for holomorphic vector bundles to cohesive modules. The organization of the paper is as follows.
	
		In section 2, we review the theory of cohesive modules and then construct the analog of Chern connections on a Hermitian cohesive module in the following sense. 
	
		\begin{theorem} Given a cohesive module $(E^{\bullet}, \mb{E}^{\dpr})$ with a Hermitian structure $h_E$, there exist an unique $\del^X$-superconnection $\mb{E}^{\dpr}$ of total degree $-1$ such that the superconnection $\mb{E} = \mb{E}^{\dpr} + \mb{E}^{\pr}$ is unitary. That is, $\mb{E}$ satisfies: 
			\begin{equation}
				(-1)^{\norm{s}} d^X h_E(s,t) = -h_E (\mb{E}s, t) + h_E(s, \mb{E}t)
			\end{equation}
		for all $s, t \in \mc{A}^{\bullet}(X,E^{\bullet})$.
		\end{theorem}
	
		We then study the basic properties of the curvature $\mc{R}^{\mb{E}}$ of the Chern superconnection $\mb{E}$.  Though it is no longer a $(1,1)$ form on $X$, if we define an exotic grading on $\mc{A}^{p, q}(X, \End^d E^{\bullet})$ by $-p + q + d$, then the curvature $\mc{R}$ is of exotic degree zero. Applying the Chern-Weil theory for superconnections, we obtain characteristic forms with values in Bott-Chern cohomology which is a refinement of deRham cohomology. We prove the image in deRham cohomology of the characteristic forms only depends on the $\mb{Z}_2$-graded topological bundle structure of a cohesive module by transgressing the characteristic forms defined by the Chern superconnection $\mb{E}$ to differential forms defined by its connection component $\mb{E}_1$.
	
		In section 3, we prove the characteristic classes in Bott-Chern cohomology are independent of the Hermitian metric by establishing several transgression formulas. These formulas were first obtained by Bott and Chern in \cite{Botta}. To generalize them to cohesive modules, we study the universal cohesive module $\widetilde{E}$ on the space of Hermitian metrics $\mc{M}$ on $E$.
	
	\begin{theorem}[Bott-Chern transgression formula]
		For any convergent power series $f(T)$, the $\mc{M}$-directional differential of the Bott-Chern characteristic form $\tr f(\mc{R}_h)$ is given by:
		\begin{equation}
			d^{\mM}\tr f(\mc{R}_h)=\del^X\delbar^X\tr \{f^{\pr}(\mc{R}_h)\cdot \theta\}
		\end{equation}
		where $\theta$ is the Maruer Cartan form on $\mc{M}$ defined by $\theta = h^{-1} \cdot d^{\mc{M}}h$. 
	\end{theorem}
	The forms $\tr\{f^{\pr}(\mc{R}_h)\cdot \theta\}$ appeared in the above double transgression formula is the holomorphic analog of Chern-Simons forms in gauge theory. In \cite{Donaldson1987}, Donaldson studied these secondary characteristic forms and their relation to Hermitian-Yang-Mills equations and stability of holomorphic Hermitian vector bundles. Following Donaldson, the technical computations in section 3 establishes the following formula which shows that $\tr\{f^{\pr}(\mc{R}_h)\cdot \theta\}$ are well defined secondary characteristic classes.
	
	\begin{theorem}[Donaldson transgression formula for secondary class] 
		For any convergent power series in $g(T)$, the $\mc{M}$-directional differential of the secondary form $\tr\{g(\mc{R}_h)\cdot \theta\}$ is given by:
		\begin{equation}
			d^{\mM}\tr \{g(\mc{R}_h)\cdot \theta \}
			=\frac{1}{2}\delbar^X\tr \{g(\Rh;[\mEhp,\theta])\cdot\theta \}
			-\frac{1}{2}\del^X\tr \{g(\Rh;[\mEdp,\theta])\cdot\theta\}
		\end{equation}
	\end{theorem}
	
	In the last section, we prove the invariance of characteristic classes under homotopy equivalences between cohesive modules. By a criteria proved by Block in \cite{Block2010}, two cohesive modules $E, F$ are homotopy equivalent if and only if there is a degree zero closed morphism $\phi \in \mc{P}_{\mc{A}^{0,\bullet}(X)}^0(E, F)$ that induces a quasi-isomorphism between the complexes $(E^{\bullet}, \mb{E}^{\dpr}_0)$ and $(F^{\bullet}, \mb{F}^{\dpr}_0)$. Using this result, we prove the invariance in two steps. First, we show that characteristic classes of an acyclic cohesive module are trivial.
	\begin{proposition}
		Assume $(E^{\bullet}, \mb{E}^{\dpr})$ is a cohesive module such that $(E^{\bullet}, \mb{E}^{\dpr}_0)$ is an acyclic complex. Let $\mb{E}^{\dpr}_t = \sum_k t^{(1-k)/2} \mb{E}^{\dpr}_k$ be the rescaled cohesive structure with parameter $t \in \mb{R}^+$ with associated curvature $\mc{R}_t$ and let $N_E$ be the grading operator on $E$. Then the integral 
		\begin{equation}
			\mc{I}_E = \int_{1}^{\infty} \tr \{ \exp(-\mc{R}_t)\cdot N_E \frac{dt}{t} \}
		\end{equation}
		is finite and we have
		\begin{equation}
			\mr{ch}(E) = \tr\exp({-\mc{R}_1})=\del^X\delbar^X \mc{I}_E
		\end{equation}
	\end{proposition}
	
	Then we show that the characteristic classes are additive with respect to the exact sequence of a mapping cone.
	
	\begin{proposition}
		If $ 0 \rightarrow E \xrightarrow{\phi} F \rightarrow \mr{Cone}(\phi)$ is the mapping cone sequence for a morphism $\phi \in \mc{P}_{\mc{A}^{0, \bullet(X)}}(E, F)$, then the Bott-Chern cohomology classes are additive in the sense that the equality
		\begin{equation}
			f(E^{\bullet}, \mb{E}^{\dpr}) - f(F^{\bullet}, \mb{F}^{\dpr}) 
			+ f(\mr{Cone}^{\bullet}(\phi), \mb{C}_{\phi}) = 0 
		\end{equation}
		holds in Bott-Chern cohomology for any convergent power series $f(T)$.
	\end{proposition}
	As a corollary, the characteristic classes defined for cohesive modules descend to $D^b_{\mr{Coh}}(X)$. In particular, this extends the Bott-Chern cohomology to coherent sheaves.
	
	\section{Hermitian Cohesive Modules and Chern superconnections}
	\subsection{Dg-category of cohesive modules}
	Let $X$ be a compact complex manifold, and let $(\mc{A}^{0, \bullet}(X),\delbar^X)$ be its Dolbeault differential graded algebra (dga). $\PA$ is the dg-category of cohesive modules over $(\Db,\delbar^X)$. We recall the definition of this dg-category from \cite{Block2010} below. Through this article, we work with double or triple $\mb{Z}$-graded objects and we write $\norm{\bullet}$ for the total degree. The commutators and traces are taken in the sense of superspaces.
	
	\begin{definition}
		A cohesive module $E=(E^{\bullet},\mb{E}^{\dpr})$ on $X$ consists of two data:
		\begin{enumerate}
			\item A finite dimensional $\mb{Z}$-graded complex vector bundle $E^{\bullet}$.
			
			\item A flat $\mb{Z}$-graded $\delbar^X$-superconnection $\mb{E}^{\dpr}$ of total degree 1.  More explicitly, $\mb{E}^{\dpr}$ is a $\mb{C}$-linear map $\mb{E}^{\dpr}: \mc{A}^{0, \bullet}(X, E^{\bullet}) \rightarrow \mc{A}^{0, \bullet}(X, E^{\bullet})$
			of total degree $1$ which satisfies both the $\delbar^X$-Leibniz formula:
			\begin{equation}\label{Leibiniz fromula of delbar superconnection}
				\mb{E}^{\dpr} e\otimes\omega = (\mb{E}^{\dpr}e)\wedge\omega + (-1)^{\norm{e}}e\otimes \delbar^X\omega
			\end{equation}
			for all $\omega \in \mc{A}^{0, \bullet}(X), e \in \mc{A}^0(X, E^{\bullet})$, and the flatness equation:
			\begin{equation}
				\mb{E}^{\dpr}\circ\mb{E}^{\dpr}=0
			\end{equation}
		\end{enumerate}
	\end{definition}
	
	\begin{remark}
		In equation \eqref{Leibiniz fromula of delbar superconnection}, we view the space of smooth sections of $E^{\bullet}$ as a right module over $\mc{A}^0(X)$. If we consider it as a left module by the isomorphism $I(e\otimes \omega) = (-1)^{\norm{e}\norm{\omega}} \omega \otimes e$, the induced superconnection $I(\mb{E}^{\dpr})$ satisfies the usual Leibniz formula: 
		\begin{equation}
			I(\mb{E}^{\dpr})(\omega \otimes e) = \delbar^X\omega \otimes e + (-1)^{\norm{\omega}} \omega \wedge I(\mb{E}^{\dpr})e
		\end{equation}
		We continue to work with the right module convention as in \cite{Block2010} so that the shift operation on the dg-category $\mc{P}_{\mc{A}^{0, \bullet}(X)}$ is simply taking $(E^{\bullet}, \mb{E})$ to $(E^{\bullet + 1}, -\mb{E})$.
	\end{remark}
	
	\begin{definition}
		The objects in the dg-category $\mc{P}_{\mc{A}^{0, \bullet}(X)}$ are cohesive modules defined above. The degree $k$ morphisms ${\mc{P}^{k}_{\mc{A}^{0, \bullet}(X)}}(E,F)$ between two cohesive modules $E=(E^{\bullet}, \mb{E}^{\dpr})$, $F=(F^{\bullet},\mb{F}^{\dpr})$ are $\mc{A}^{0, \bullet}(X)$-linear maps $\phi: \mc{A}^{0, \bullet}(X, E^{\bullet}) \rightarrow  \mc{A}^{0, \bullet}(X, F^{\bullet})$ of total degree $k$. The differential $d: \mc{P}^{k}_{\mc{A}^{0, \bullet}(X)} (E,F)\rightarrow \mc{P}^{k+1}_{\mc{A}^{0, \bullet}(X)}(E,F)$ is defined by the commutator:
		\begin{equation}
		d(\phi) = \mb{F}^{\dpr}\circ\phi-(-1)^k\phi\circ\mb{E}^{\dpr}
		\end{equation}
		It's simple to verify $d^2=0$ and therefore $\mc{P}_{\mc{A}^{0, \bullet}(X)}$ is a dg-category.
	\end{definition}
	
	\begin{example}
		If $(E^{\bullet}, \delta)$ is a complex of holomorphic vector bundles on $X$, it defines a cohesive module with cohesive structure defined by setting $\mb{E}^{\dpr}_0 = \delta$, $\mb{E}^{\dpr}_1 = (-1)^{\bullet}\delbar^{E^{\bullet}}$ and $\mb{E}^{\dpr}_k = 0$ for $k > 1$. The flatness condition is equivalent to the following set of equations: 
		\begin{align}
		&\delta \circ \delta=0   &(\delta \textrm{ defines a complex of vector bundles}) \\
		&\delta \circ \delbar^{E^{\bullet}} = \delbar^{E^{\bullet + 1}} \circ \delta 
		&(\delta \textrm{ are holomoprhic homomorphisms}) \\
		&\delbar^{E^{\bullet}} \circ \delbar^{E^{\bullet}} = 0
		&(\delbar^{E^{\bullet}}\ \textrm{defines a holomorphic structure on } E^{\bullet})
		\end{align}
	\end{example}
	
	\begin{remark}
		More generally, Block proved that the homotopy category of $\mc{P}_{\mc{A}^{0, \bullet}(X)}$ is equivalent to $D^b_{\mr{Coh}}(X)$ in \cite{Block2010}. That is, any complex of sheaves of $\mc{O}_X$-modules with coherent cohomology is represented by a cohesive module. In particular, any coherent sheaf is represented by a cohesive module $(E^{\bullet}, \mb{E}^{\dpr})$, unique up to homotopy, such that the underlying complex $(E^{\bullet}, \mb{E}^{\dpr}_0)$ is exact except at the end. 
	\end{remark}
	
	\begin{definition}
		A Hermitian metric $h$ on a cohesive module $E$ is a Hermitian metric on the $\mb{Z}$-graded bundle $E^{\bullet}$ such that the $E^{k}$ is orthogonal to $E^{l}$ if $k \neq l$. We denote by $(E, \mb{E}^{\dpr}, h)$ a Hermitian cohesive module.
	\end{definition}
	
	We define a conjugate linear involution on differential forms that is different from complex conjugation. It is motivated by the involution on Clifford algebras.
	
	\begin{definition}
		For $\omega \in \mc{A}^{k}(X)$ a complex differential form in degree $k$,  we define
		\begin{equation}
		\omega^*=(-1)^{\frac{(k+1)k}{2}}\overline{\omega}
		\end{equation}
	\end{definition}
	
	The following lemma summarizes the properties of $*$ and it is simple to verify them by the definition.
	\begin{lemma}\label{properties of star on forms}
		The $*$-operation is the unique operation on $\mc{A}^{\bullet}(X)$ such that $f^*=\bar{f}$ if $f$ is a smooth function,  $\omega^*=-\overline{\omega}$ if $\omega$ is an one form and in general,
		\begin{equation}
			(\omega\wedge\eta)^*=\eta^*\wedge\omega^*, \ \forall \omega, \eta \in \mc{A}^{\bullet}(X)
		\end{equation}
	\end{lemma}
	
	\begin{definition}
		We extend the $*$-operation to  $\mc{A}^{\bullet}(X,\End^{\bullet}E)$. If $L\otimes \omega$ is a homogeneous element, we define:
		\begin{equation}
		(L\otimes\omega)^*=(-1)^{\norm{L}\norm{\omega}}L^*\otimes\omega^*
		\end{equation}
		and extend it to $\mc{A}^{\bullet}(X,\End^{\bullet}E)$ by linearity.
	\end{definition}

	\begin{lemma}\label{properties of star on End valued forms}
		The operator $*$ extends to a conjugate linear involution on the algebra $\mc{A}^{\bullet}(X,\End^{\bullet}E)$.
	\end{lemma}
	
	\begin{definition} 
		If $h$ is a Hermitian structure on the cohesive module $E^{\bullet}$, then we extend $h$ to $\mc{A}^{\bullet}(X, E^{\bullet})$ by the following formula on homogeneous elements:
		\begin{equation}
		h(e \otimes \omega, f \otimes \eta) = \omega^* \wedge h(e,f) \wedge \eta
		\end{equation}
		where $e,f\in \mc{A}^0(X, E^{\bullet})$ and  $\omega,\eta \in \mc{A}^{\bullet}(X)$ and extend it to $\mc{A}^{\bullet}(X, E^{\bullet})$ by linearity.
	\end{definition}
	
	The following lemma shows that the $*$ operation defines the adjoint operation on $\mc{A}^{\bullet}(X, \End E)$ with respect to $h$.
	\begin{lemma}\label{linear terms in Chern connection}
		If $A \in \mc{A}^{p, q}(X, \End^{d}E)$, then $A^* \in \mc{A}^{q, p}(X, \End^{-d}E)$ and for any $s, t \in \mc{A}^{\bullet}(X, E^{\bullet})$, we have
		\begin{equation}
		h(As, t) = h(s, A^*t)
		\end{equation}  
	\end{lemma}
	\begin{proof}
		Without loss of generailty, we assume $A = L\otimes \tau, s = e\otimes\omega, t = f\otimes\eta$ are homogeneous. In addition, we assume $\norm{e} + \norm{L} = \norm{f}$ for otherwise both side of the equation are zero and equality holds. Under these assumptions, we may compute as:
		\begin{align*}
		h(As,t)
		=&h(L\otimes\tau(e\otimes\omega),f\otimes\eta)
		=(-1)^{\norm{\tau}\norm{e}}h(Le\otimes \tau\wedge\omega,f \otimes \eta)\\
		=&(-1)^{\norm{\tau}\norm{e}}(\tau\wedge\omega)^*h(Le,f)\eta
		=(-1)^{\norm{\tau}\norm{e}}\omega^* \wedge \tau^*h(e,L^*f)\eta\\
		=&(-1)^{\norm{\tau}\norm{e}}\omega^*h(e,L^*f) \tau^*  \wedge \eta
		=(-1)^{\norm{\tau}(\norm{e}+\norm{f})}\omega^*h(e,L^*\otimes\tau^*(f\otimes\eta))
		\end{align*}
		where for the third fourth equality, we used Lemma \ref{properties of star on forms}. By assumption, $\norm{e} + \norm{L} = \norm{f}$, so we have
		$\norm{\tau}(\norm{e}+\norm{f}) = \norm{\tau}\norm{L} \mr{\ mod \ }2$. Finally, since $A^{*} = (-1)^{\norm{L}\norm{\tau}}L^* \otimes \tau^*$ by definition, we have
		\begin{equation}
		h(As,t) = (-1)^{2\norm{L}\norm{\tau}}h(s,A^*t) = h(s, A^*t)
		\end{equation}
	\end{proof}
	
	\subsection{Existence and Uniqueness of Chern superconnection on a Hermitian cohesive module}
	The holomorphic Chern connection $\nabla^E$ for a holomorphic Hermitian vector bundle $E$ is the unique unitary connection on $E$ with $\delbar^X$-component given by the Dolbeault differential $\delbar^E$. We generalize this construction in this section to Hermitian cohesive modules.
	
	\begin{lemma}\label{Chern connection for single bundle}
		For a complex Hermitian vector bundle $(E, h)$ with an $\delbar^X$-connection $\nabla^{\dpr}$ which is not necessarily flat, there exist an unique $\del^X$-connection $\nabla^{\pr}$ such that 
		$\nabla =\nabla^{\pr} + \nabla^{\dpr}$ is unitary with respect to $h$. That is, for any $e, f \in \mc{A}^{\bullet}(X, E)$, $\nabla$ satisfies:
		\begin{equation}\label{unitary connection.1}
			d^Xh(e,f) = -h(\nabla{e},f) + h(e,\nabla{f})
		\end{equation}
	\end{lemma}
	
	\begin{proof}
		The problem is local, so it suffices to construct $\nabla$ and prove its uniqueness locally. Choose a local frame $s=(s_1,s_2,...s_n)$ on $E$, let $\Theta \in \mc{A}^{0,1}(X,\End{E})$ be the connection $(0, 1)$-form
		such that $\nabla^{\dpr}s = s \otimes \Theta$. With respect to the frame $s$, the Hermitian metric $h$ is represented by the Hermitian matrix valued function $\mc{H}=h(s,s)$. Since any $(1,0)$-connection $\nabla^{\pr}$ is locally represented by its connection $(1,0)$-form $\Omega \in \mc{A}^{1,0}(X,\End{E})$ such that $\nabla^{\pr}s=s \otimes \Omega$, if we set $\nabla=\nabla^{\pr} + \nabla^{\dpr}$, the condition for $\nabla$ being unitary is equivalent to:
		\begin{equation}\label{unitary connection.2}
			d^X \mc{H} = -(\Theta+\Omega)^{*}\mc{H} + \mc{H}(\Theta+\Omega)
		\end{equation}
		Comparing the $(1,0)$ and $(0,1)$ component of equation (\ref{unitary connection.2}) above, we see (\ref{unitary connection.1}) is equivalent to 
		\begin{equation}\label{unitary connection.3}
			\del^X\mc{H} = -\Theta^*\mc{H} + \mc{H}\Omega
		\end{equation}
		\begin{equation}\label{unitary connection.4}
			\delbar^X\mc{H} = -\Omega^*\mc{H} + \mc{H}\Theta
		\end{equation}
		Using (\ref{unitary connection.3}), we can solve for $\Omega$ as
		\begin{equation}\label{unitary connection.5}
			\Omega=\mc{H}^{-1} \del^X \mc{H} + \mc{H}^{-1} \Theta^*\mc{H}
		\end{equation}
		which shows the uniqueness of $\nabla^{\pr}$. But we can also use (\ref{unitary connection.5}) as the definition of $\Omega$. It remains to verify that with $\Omega$ so defined, equation (\ref{unitary connection.4}) is satisfied. Since $\mc{H}$ is a Hermitian matrix,  $\mc{H}^*=\mc{H}$ and we can now compute $\Omega^*$ as:
		\begin{align*}
		\Omega^*
		=&(\del^X\mc{H})^*(\mc{H}^{-1})^*+\mc{H}^{*}\Theta^{**}(\mc{H}^{-1})^*\\
		=&-\delbar^X\mc{H}^*(\mc{H}^*)^{-1}+\mc{H}^*\Theta(\mc{H}^*)^{-1}\\
		\numberthis\label{unitary connection.6}
		=&-(\delbar^X\mc{H})\mc{H}^{-1} + \mc{H}\Theta\mc{H}^{-1}
		\end{align*}
		Multiplying $\mc{H}$ on the right, we get \eqref{unitary connection.4}.
	\end{proof}
	
	\begin{remark}
		The above equation \eqref{unitary connection.1} differs by a minus sign from the ordinary equation for unitary connections:
		\begin{equation}{\label{ordinary unitary connection}}
			d^Xh(e,f)=h(\nabla e,f)+h(e,\nabla f)
		\end{equation}
		This is caused by our extension of $h$ to $\mc{A}^{\bullet}(X, E^{\bullet})$. In the ordinary case, if $s\otimes \Omega$ is a one form with values in $E$, then $h(s\otimes \Omega, t) = \overline{\Omega}h(s, t)$. However, by our definition,  $h(s\otimes \Omega, t) = \Omega^*h(s, t) = -\overline{\Omega}h(s, t)$.
	\end{remark} 
	
	\begin{definition}
		We define an exotic grading on the spaces   $\mc{A}^{\bullet}(X,\End^{\bullet}E)$ and $\mc{A}^{\bullet}(X,E^{\bullet})$. For an element $A=L\otimes\tau$, where $L\in \End^d(E^\bullet)$ or $L \in E^d$, and $\tau \in \mc{A}^{p,q}(X)$, define its exotic degree by $\deg(A)= -p+q+d$. $\mc{A}^{\bullet}(X)$ inherit an exotic grading as well given by $\deg(\tau) = - p + q$. We denote by $\mc{G}^{k}$ the subspaces of exotic degree $k$. In particular, the subspace $\mc{G}^0$ in $\mc{A}^{\bullet}(X)$ consists of forms with bi-degree $(p, p)$.  
	\end{definition}
	
	\begin{lemma} 
		With respect to the exotic grading, $\mc{A}^{\bullet}(X)$ is a $\mb{Z}$-graded algebra; $\mc{A}^{\bullet}(X, \End^{\bullet} E)$ is a $\mb{Z}$-graded algebra and $\mb{Z}$-graded module over $\mc{A}^{\bullet}(X)$; $\mc{A}^{\bullet}(X, E^{\bullet})$ is both a $\mb{Z}$-graded module over $\mc{A}^{\bullet}(X)$ where the action is exterior product, and a $\mb{Z}$-graded module over $\mc{A}^{\bullet}(X, \End^{\bullet} E)$ where the action is evaluation. In addition, the $*$-operator defined before on $\mc{A}^{\bullet}(X)$ and $\mc{A}^{\bullet}(X, \End^{\bullet}E)$ interchanges $\mc{G}^{\bullet}$ with $\mc{G}^{-\bullet}$.
	\end{lemma}
	\begin{proof}
		It's simple to verify that the subspaces $\mc{G}^{\bullet}$ respects all these structures in the sense that
		\begin{equation}
			\mc{G}^{i}\cdot \mc{G}^j \subseteq \mc{G}^{i+j}
		\end{equation}
		holds whenever the composition is defined, either by the algebra multiplication or module action.
	\end{proof}
	
	\begin{proposition}[Chern superconnection]\label{Chern superconnection existence}
		Let $E = (E^{\bullet}, \mb{E}^{\dpr}, h)$ be a Hermitian cohesive module. There exist an unique $\del^X$-superconnection
		$\mb{E}^{\pr}: \mc{A}^{0}(X, E^{\bullet}) \rightarrow \mc{A}^{\bullet,0}(X, E^{\bullet})$ of exotic degree $-1$
		such that the superconnection $\mb{E}=\mb{E}^{\pr}+\mb{E}^{\dpr}$ is unitary. That is for any $ s,t \in \mc{A}^{\bullet}(X, E^{\bullet})$, we have
		\begin{equation}\label{unitary superconnection}
		(-1)^{\norm{s}}d^X h(s,t)=-h(\mb{E}s,t)+h(s,\mb{E}t)
		\end{equation}
	\end{proposition}
	
	\begin{proof}
		There is a decomposition of $\mb{E}^{\dpr} = \mb{E}^{\dpr}_1 + \sum_{k \neq 1} \mb{E}^{\dpr}_k$ into an ordinary $\delbar^X$-connection $\mb{E}^{\dpr}_1$ on $E^{\bullet}$ and linear terms $\mb{E}^{\dpr}_k \in \mc{A}^{0, k}(X, \End^{1-k}E)$. 
		
		By lemma \ref{Chern connection for single bundle}, there is an unique $\del^X$-connection $\mb{E}^{\pr}_1$ on $E^{\bullet}$ such that $\nabla^{E^{\bullet}} = \mb{E}^{\dpr}_1 + \mb{E}^{\pr}_1$ is unitary on the graded Hermitian complex vector bundle $E^{\bullet}$. If we set $\mb{E}^{\pr}_k = (\mb{E}^{\dpr}_k)^*$ for $k \neq 1$, by lemma \ref{linear terms in Chern connection}, $\mb{E}_k = \mb{E}^{\dpr}_k + \mb{E}^{\pr}_k$ is unitary. Adding $\nabla^{E^{\bullet}}$ and $\mb{E}_k$, the resulting $d^X$-superconnection $\mb{E}$ is unitary. Uniqueness is obvious from the construction and equation \eqref{unitary superconnection}.
	\end{proof}
	
	\begin{remark}
		We can always write a homogeneous term $A=L\otimes\tau$ in such a way that $\tau$ is a real form of degree $k$. Since $A$ is odd, $A^*$ is just $L^*\otimes \tau^*$ and therefore $A = A^*$ is equivalent to $L=(-1)^{\frac{k(k+1)}{2}}L^*$. In other words, $L$ satisfies the following conditions:	
		$$\begin{cases}
		L \textrm{ is Hermitian }  &k=0,3\  (\textrm{mod } 4)\\
		L \textrm{ is skew Hermitian } &k=1,2\  (\textrm{mod } 4)\\
		\end{cases}$$
		In particular, this is compatible with the well known fact that an unitary connection has skew Hermitian connection matrix.
	\end{remark}

	\subsection{Chern-Weil construction of characteristic forms}
	
	\begin{definition}
		Let $\mb{E}$ be the superconnection for a Hermitian cohesive module $(E^{\bullet}, \mb{E}^{\dpr}, h)$ given by Proposition \ref{Chern superconnection existence}. We call it the Chern superconnection. The curvature of $\mb{E}$ is defined by the usual formula 
		\begin{equation}
		\mc{R} = \mb{E}^2 = \frac{1}{2}[\mb{E}, \mb{E}]
		\end{equation}
	\end{definition}
	
	In the classical case of Hermitian holomorphic vector bundles, the curvature of the Chern connection is a $(1,1)$-form. In the case of Hermitian cohesive modules, we have similar properties of Chern superconnections for cohesive modules.
	
	\begin{lemma}\label{curvature is degree 0}
		If $\mb{E}$ is the Chern superconnection for a Hermitian cohesive module $(E^{\bullet}, \mb{E}, h)$ and $\mc{R}$ is its curvature, then it satisfies the following properties:
		\begin{align}
		&\mc{R}^* = \mc{R}
		\\
		&\mb{E}^{\pr}\circ\mb{E}^{\pr} = 0
		\\
		&\mc{R}=[\mb{E}^{\pr},\mb{E}^{\dpr}]
		\end{align}
		Consequently, the curvature $\mc{R}$ is of exotic degree zero $\mc{R} \in \mc{G}^0$.
	\end{lemma}
	
	\begin{proof} 
		By definition, the curvature is
		\begin{align*}
			\mc{R}
			= &\mb{E}\circ \mb{E}
			= (\mb{E}^{\pr} + \mb{E}^{\dpr})\circ (\mb{E}^{\pr} + \mb{E}^{\dpr}) \\
			= &(\mb{E}^{\pr})^2 + [\mb{E}^{\pr},\mb{E}^{\dpr}] + (\mb{E}^{\dpr})^2 \\
			= &(\mb{E}^{\pr})^2 + [\mb{E}^{\pr},\mb{E}^{\dpr}]
		\end{align*}
		where the term $(\mb{E}^{\dpr})^2$ vanishes by our flatness assumption on $\mb{E}^{\dpr}$. Since the exotic degree of $\mb{E}^{\pr}$ is $-1$ and that of $[\mb{E}^{\pr},\mb{E}^{\dpr}]$ is 0, $(\mb{E}^{\pr})^2 \in \mc{G}^{-2}$ and $[\mb{E}^{\pr},\mb{E}^{\dpr}]\in \mc{G}^0$.
		We use equation \eqref{unitary superconnection} repeatedly to get:
		\begin{align*}
			0 
			= &(d^X)^2{h}(s, t) = (-1)^{\norm{s} + 1} d^X h(\mb{E}s, t) + (-1)^{\norm{s}}d^X h(s, \mb{E}t)\\
			=&-h(\mc{R}s, t) + h(\mb{E}s,\mb{E}t) - h(\mb{E}s, \mb{E}t) + h(s, \mc{R}t)\\
			=&-h(\mc{R}s, t) + h(s,\mc{R}t)
		\end{align*}
		Since $\mc{R}$ is linear, by Lemma \ref{linear terms in Chern connection}, the curvature $\mc{R}$ is self adjoint $\mc{R}=\mc{R}^*$. Since $*$ maps $\mc{G}^{-2}$ to $\mc{G}^2$, we see $(\mb{E}^{\pr})^2 = 0$. The last equality follows from this. 
	\end{proof}
	
	\begin{lemma}[Bianchi Identity]\label{Bianchi}
		$[\mE,\mc{R}]=0$
	\end{lemma}
	\begin{proof}
		By definition, $\mc{R}=\frac{1}{2}[\mE,\mE]$. Use either the graded Jacobi identity or simply expand the expression, we have
		\begin{align*}
		[\mE,\mc{R}] = [\mE, \mE^2]
		=\mE\cdot\mE^2-\mE^2\cdot\mE=0	
		\end{align*}
	\end{proof}
	
	\begin{definition}
		Assume $(E, \mE^{\dpr}, h)$ is a Hermitian cohesive module with Chern superconnection $\mb{E}_h$ and curvature $\mc{R}_h$. Let $f(T)$ be a convergent power series in $T$, we define the characteristic form of $(E, h)$ associated to $f(T)$ by
		\begin{equation}
		f(E^{\bullet}, \mb{E}^{\dpr}, h) = \tr f(\mc{R}_h)
		\end{equation}
		where $\tr$ is the supertrace in the sense of Quillen.
	\end{definition}
	
	\begin{remark}
		The power series is required to be convergent since $\mc{R}$ is no longer concentrated in degree $(1, 1)$. In practice, this can be relaxed. For example, if $\mb{E}^{\dpr}_0 = 0$ and therefore $\mc{R} \in \mc{A}^{\bullet > 0}$, we can define the characteristic forms by any power series since $\mc{R}$ is nilpotent. 
	\end{remark}

	\begin{proposition}\label{characteristic forms are closed}
		The characteristic forms are closed.  
		\begin{equation}
		d^Xf(E^{\bullet}, \mb{E}^{\dpr} , h) = 0
		\end{equation}
	\end{proposition}
	\begin{proof}
		Without loss of generality, we assume $f(T)=T^n$ is a monomial. Then we have
		\begin{equation}
			d^X \tr(\mc{R}^n_h) = \sum_i\tr (\mc{R}_h^{i-1}[\mE_h, \mc{R}_h]\mc{R}^{n-i}_h) 
		\end{equation} 
		Since each term in the summation is zero by Bianchi identity, $\tr(\mc{R}^n_h)$ is closed.
	\end{proof}
	
	Consequently, the characteristic forms defines deRham cohomology classes. We will show that these cohomology classes only dependent on the connection component $\nabla^E = \mb{E}_1$ in $\mb{E}_h$. Let $\mE_t=(1-t){\nabla^E}+t\mE$ be a one parameter family of superconnections that joins $\nabla^E$ to $\mE$. If we write $A=\mE-\mE_1=\sum_{k\neq 1} \mE_k$ for the sum of linear terms, then $\mE_t = \nabla^E + tA$ and $\mc{R}_t$ are the corresponding curvatures for $\mE_t$. 
	
	\begin{lemma} The deformation of curvature is computed by:
		\begin{equation}
			\frac{d}{dt}\mc{R}_t=[\mE_t,A]
		\end{equation}
	\end{lemma}
	\begin{proof}
		Since $\mE_t = \nabla^E + tA$, we have $\mc{R}_t = \mc{R}_0 + t[\nabla^E, A] + t^2 A^2$. We can then compute
		\begin{equation}
		\frac{d}{dt}\mc{R}_t = [\nabla^E, A] + 2tA^2 
		\end{equation}
		since $2tA^2 = t[A, A] = [tA, A]$, the right side of the above equality is just $[\mb{E}_t, A]$.
	\end{proof}
	
	\begin{corollary}\label{deRham class of characteristic forms}
		If we define $f(E, \mE_t)$ by the same formula $\tr f(\mc{R}_t)$, we have
		\begin{equation}{\label{linear transgression}}
		f(E, \mb{E}) - f(E, \nabla^E)
		=d^X \int_{0}^{1} \tr \{ A\cdot f^{\pr}(\mc{R}_t) \} dt 
		\end{equation}
		Consequently, $f(E, \mb{E})$ is cohomologous to  $f(E, \nabla^E)$ in deRham cohomology.
	\end{corollary}
	It is well-known that $f(E, \nabla^E)$ only depends on the topological vector bundle structure of $E$ as a $\mb{Z}_2$-graded vector bundle, so the characteristic class in deRham cohomology defined by $f(E^{\bullet}, h, \mb{E}^{\dpr})$ is independent of $h$. To capture the intrinsic holomorphic structures of cohesive modules, we introduce Bott-Chern cohomology which refines deRham cohomology.
	
	\begin{definition}
		Let $\mc{A}^{p,p}(X)$ be the space of forms of bi-degree $(p, p)$ and $Z^p(X)$ be the subspace of $d^X$-closed forms. We set $B^p(X)$ the subspace of $Z^p(X)$ that is the image of $\mc{A}^{p-1, p-1}(X)$ under $\del^X\delbar^X$, then the $p$-th Bott-Chern cohomology $H_{BC}^p(X)$ is defined by: 
		\begin{displaymath}
		H_{BC}^p(X) = \frac{Z^p(X)}{B^p(X)}
		\end{displaymath}
	\end{definition}
	
	By Lemma \ref{curvature is degree 0}, $f(E, \mE^{\dpr}, h) \in \mc{G}^0$ and hence the forms defines Bott-Chern cohomology classes. We will prove in the next section that these refined cohomology classes are independent of the Hermitian structure $h$ as well. In the last section, we will show they are further invariant under homotopy equivalences between cohesive modules. 
	
	\section{Transgression formulas for characteristic forms}
	\subsection{Universal cohesive module and Maurer-Cartan equation}
	We fix a cohesive module $E=(E^{\bullet},\mb{E}^{\dpr})$ on $X$ and let $\mc{M}$ be the space of Hermitian structures on $E^{\bullet}$. We endow it with the topology of uniform $\mc{C}^{\infty}$-convergence on compact subsets of X. If $h\in \mc{M},x \in X$, let $A_{x}^{h}$ be the subset of $\End^0{E_x}$ whose elements are Hermitian endomorphisms with respect to $h_x$. $A^h_x$ forms a bundle $A^h$ over $X$, and the tangent space of $\mc{M}$ at $h$, $T_h\mc{M}$, can be identified with the linear space of sections of $A^{h}$ over $X$. 
	
	\begin{definition}
		Consider the projection $\pi: \mc{M} \times X \rightarrow X$, we define $\widetilde{E}^{\bullet}$ as the pull-back bundle $\pi^*E^{\bullet}$ over $\mc{M} \times X$. There is an universal cohesive structure $\widetilde{\mb{E}}^{\dpr}$ on $\widetilde{E}^{\bullet}$ defined by
		\begin{equation}
		\widetilde{\mb{E}}^{\dpr} = \delbar^{\mc{M}} + \mb{E}^{\dpr}
		\end{equation}
		In addition, we write $\tilde{h}$ for the universal Hermitian form on $\widetilde{E}^{\bullet}$. We denote by $\widetilde{E}$ the universal Hermitian cohesive module $(\widetilde{E}^{\bullet}, \widetilde{\mb{E}}^{\dpr}, \tilde{h})$ over $\mM \times X$.
	\end{definition}
	
	\begin{remark}
		The bundle $\widetilde{E}^{\bullet}$ is flat in the $\mc{M}$-direction with flat connection $d^{\mc{M}}$. In the above formula, $\delbar^{\mc{M}}$ acts on an element $e\otimes \omega$ by $(-1)^{\norm{e}}e\otimes \delbar^{\mc{M}}\omega$.
	\end{remark}
	
	\begin{lemma}
		$\widetilde{\mb{E}}^{\dpr}$ defines a $\mb{Z}$-graded flat $\delbar^{\mc{M}\times X}$-connection on $\widetilde{E}^{\bullet}$.
	\end{lemma}
	
	\begin{proof}
		Note that the vertical cohesive structure $\mb{E}^{\dpr}$ is constant along $\mc{M}$-direction, so $d^{\mc{M}}\mb{E}^{\dpr} = 0$, we have
		\begin{equation}
		(\widetilde{\mb{E}}^{\dpr})^2 = (\delbar^{\mc{M}} + \mb{E}^{\dpr})^2 = (\delbar^{\mc{M}})^2 + \delbar^{\mc{M}}(\mb{E}^{\dpr}) + (\mb{E}^{\dpr})^2 = 0
		\end{equation}   
	\end{proof}
	
	\begin{definition}{\label{MC form}}
		The Maurer-Cartan one form $\theta \in \mc{A}^{1}(\mc{M}\times X, \End^{0}\widetilde{E})$ on $\mc{M}$ with values in $A$ is the one form defined by 
		\begin{equation}\label{theta}
		\theta=h^{-1}d^{\mc{M}}{h}
		\end{equation}
		where $d^M$ is the exterior differential on $\mc{M}$. 
	\end{definition}
	
	\begin{lemma}
		$\theta$ is a one form in the $\mc{M}$-direction with values in Hermitian endomorphisms, $\theta \in \mc{A}^1(\mc{M}, \End(E, h))$. 
	\end{lemma}
	\begin{proof}
		We choose a frame $(e_1, e_2,...,e_n)$ for $E^{\bullet}$ locally on $X$. Any Hermitian form $h$ is represented by a smooth function with values in Hermitian matrix by :
		\begin{equation}
			\mc{H}: x\in X \rightarrow [h_x(e_i,e_j)]_{ij}
		\end{equation}
		If $\mc{H}_t(x)$ is a smooth family of Hermitian matrix valued functions with parameter $t$, we can differentiate $\mc{H}_t(X)$ with respect to $t$, then $\mc{H}_t^{-1}\dot{\mc{H}_t}$ is a Hermitian with respect to $\mc{H}_t$. 
	\end{proof}
	
	\begin{lemma}\label{MC-equation}
		$\theta$ satisfies the Maurer-Cartan equation:
		\begin{equation} 
		d^{\mM}\theta=-\theta^2
		\end{equation}
	\end{lemma}
	
	\begin{proof} 
		This is a straightforward calculation:
		\begin{align}
		d^{\mc{M}}\theta
		=&d^{\mM}(h^{-1}d^{\mM}h)=d^\mM(h^{-1})d^\mM h+h^{-1}d^{\mM}(d^{\mM}h)\\
		=&-h^{-1}(d^{\mM}h)h^{-1}d^{\mM}h=-\theta^2
		\end{align}
	\end{proof}
	
	Since $\widetilde{E}$ is flat along the $\mc{M}$-direction, by the explicit expression of $\theta$ with respect to a local frame on $X$, we see $\theta$ measures the deformation of $h$ over $\mc{M}$ in the following sense.
	\begin{lemma} 
		The flat $\mM$-directional exterior derivative $d^{\mM}$ has the following compatibility relation with the universal Hermitian form $\tilde{h}$.
		\begin{equation}\label{unitary d and MC form}
		(-1)^{\norm{s}}d^{\mc{M}}\tilde{h}(s,t)
		= -\tilde{h}(d^{\mc{M}}s,t) 
		+ \tilde{h}(s,\theta t) 
		+ \tilde{h}(s,d^{\mc{M}}t)
		\end{equation}
		for all $s, t \in \mc{A}^{0}(\mc{M}\times X, \widetilde{E})$.
	\end{lemma}  

	The above lemma together with equation \eqref{unitary superconnection} for each $h\in \mc{M}$ shows that $\widetilde{\mE} = d^{\mc{M}} + \mb{E}_h$ differs from the Chern superconnection of $(\tE,\tilde{h},\tmE^{\dpr})$ by the Maurer-Cartan form $\theta$.
	
	\begin{proposition} The universal superconnection $\tmE$ satisfies the following equation
		\begin{equation}\label{universal unitary superconnection}
		(-1)^{\norm{s}}d^{\mc{M} \times X}\tilde{h}(s,t)
		= -\tilde{h}(\tmE s,t) + \tilde{h}(s,\tmE t) + \tilde{h}(s,\theta t)
		\end{equation}
	\end{proposition}
	
	Even though the previous proposition shows that $\tmE$ is not the Chern superconnection, we will use $\tmE$ to study $\tE$.
	Denote by $\tR$ the curvature of $\tmE$, it is given by
	\begin{equation}\label{universal curvature}
	\tR = ({d^{\mM}}+\mE_h) \circ ({d^{\mM}}+\mE_h) = {d^{\mM}}\mE_h + \mc{R}_h
	\end{equation}
	
	\begin{proposition}\label{derivative of connection} The $\mM$-directional derivative of the Chern superconnection $\mE_h$ is given by
		\begin{equation}\label{derivative of connection.1}
		{d^{\mM}}\mE_h = -[\mE^{\pr}_h,\theta]
		\end{equation}  
		and the $\mM$-directional derivative of curvature $\tR$ is given by
		\begin{equation}
		d^{\mc{M}}\tR = d^{\mM}\mc{R}_h 
		\end{equation}	
	\end{proposition}
	
	\begin{proof}
		Since $\mE_h=\mE_h^{\pr} + \mE^{\dpr}$ and $\mE^{\dpr}$ is independent of $h$, we have
		\begin{equation}
		{d^{\mM}}\mE_h = d^{\mM}\mE_h^{\pr}
		\end{equation}
		By the explicit construction of $\mE^{\pr}_h$ from $\mE^{\dpr}$ and $h$, we can write
		\begin{equation}
		\mE^{\pr}_h=h^{-1}\circ \mE^{\dpr}\circ h
		\end{equation}
		Taking exterior differential in $\mc{M}$-variable, we have
		\begin{align*}
		d^{\mc{M}}\mE^{\pr}_h
		& = d^{\mc{M}}h^{-1} \circ \mb{E}^{\pr} \circ h-h^{-1}\circ\mE^{\pr} \circ d^{\mM}h\\
		&=-h^{-1}\circ d^{\mM}h\circ h^{-1}\circ\mE^{\pr}\circ h-h^{-1} \circ \mE^{\pr} \circ h \circ h^{-1} \circ d^{\mM}h \\
		&=-\theta\circ\mE^{\pr}_h-\mE^{\pr}_h\circ\theta 	
		=-[\mE^{\pr}_h,\theta]
		\end{align*}
		The second equality is a simple consequence of equation \eqref{universal curvature}.
	\end{proof}

	\subsection{Bott-Chern double transgression formula}
	Recall the characteristic forms defined in the previous section $f(E^{\bullet}, \mb{E}^{\dpr},h)=\tr f(\mc{R}_h)$ for a convergent power series $f(T)$, the following transgression formula computes the deformation of the characteristic forms over $\mc{M}$.
	
	\begin{proposition}[First Transgression Formula]
		\label{first transgression}
		\begin{equation}\label{1st transgress delbar}
		d^{\mM}\tr f(\Rh)
		= -\delbar^X\tr\{f^{\pr}(\mc{R}_h) \cdot [\mE_h^{\pr},\theta]\}
		\end{equation}
		\begin{equation}\label{1st transgress del}
		\del^X\tr\{ f^{\pr}(\mc{R}_h) \cdot [\mE_h^{\pr},\theta] \} = 0
		\end{equation}
	\end{proposition}
	
	\begin{proof}
		On the universal Hermitian cohesive module $\tE$, by proposition \ref{characteristic forms are closed} and equation \eqref{universal curvature}, we have
		\begin{equation}\label{transgression.2}
		0 = d^{\mM\times X}\tr{f(\tR)} 
		= d^{\mM\times X}\tr{f(\mc{R}_h + d^{\mM}\mE_h)}
		\end{equation}
		Without loss of generality, we assume $f(T) = T^n$ and we expand (\ref{transgression.2}) by the multilinear property of $\tr f(T)$ to get: 
		\begin{equation}
		\label{transgression.3}
		0 = d^X \tr \mc{R}^n_h + d^{\mc{M}} \tr \mc{R}^n_h + d^{X}\sum_i \tr(\mc{R}_h^i \cdot d^{\mM}\mE_h \cdot \mc{R}_h^{n-1-i}) + ... 
		\end{equation}
		The first term is zero by proposition \ref{characteristic forms are closed} and the omitted terms are at least degree $2$ forms in $\mc{M}$-variables. If we collects the forms of degree one in the $\mM$-variables, we have
		\begin{equation}\label{transgression.4}
		0 = d^{\mM} \tr(\mc{R}_h^n) + n\cdot d^{X}\tr(\mc{R}_h^{n-1} \cdot {d^{\mM}}\mE_h)
		\end{equation}
		where we commute $\mc{R}_h^i \cdot d^{\mM}\mE_h$ and $\mc{R}_h^{n-1-i}$ under $\tr$. Using Lemma \ref{derivative of connection} $d^{\mM}\mb{E}_h = -[\mb{E}^{\pr}_h, \theta]$, we have the following equality:
		\begin{equation}\label{transgression.5}
		d^{\mM}\tr f(\mc{R}_h)
		= d^{X}\tr\{f^{\pr}(\mc{R}_h) \cdot [\mE_h^{\pr},\theta]\}
		\end{equation}
		
		Now we compare both sides of (\ref{transgression.5}) and the consider the subspaces $\mc{G}^{\bullet}$ defined by the exotic degree.  On the left side of (\ref{transgression.5}), we have a one form on $\mc{M}$ with values in $\mc{G}^0$ since $\mc{R}_h \in \mc{G}^0$ by Lemma \ref{curvature is degree 0}. On the right hand side, $\mc{R}_h$ is in $\mc{G}^0$ while $[\mE_h^{\pr},\theta]$ is in $\mc{G}^{-1}$ since $\mE_h^{\pr}$ is of exotic degree $-1$ and $\theta$ is of exotic degree $0$. Finally,  since $\del^X$ increases the exotic degree by $-1$ while $\delbar^X$ increases the exotic degree by $1$, we get equation \eqref{1st transgress delbar} by comparing the $\mc{G}^0$ component and equation \eqref{1st transgress del} by the $\mc{G}^{-2}$ component.
	\end{proof}
	
	Our next goal is to express $\tr\{f^{\pr}(\mc{R}_h) \cdot [\mE_h^{\pr},\theta]\}$ in equation \eqref{1st transgress del} as the image of $\del^X$. To do this, we first introduce a notation.
	\begin{definition}
		If $g(T)=T^n$, for a pair $(A;B)$ of variables, we define
		\begin{equation}
		g(A;B)=\sum_{i=1}^{n}A^{i-1}BA^{n-i}
		\end{equation} 
		In general for a convergent power series $g(T)$, we define $g(A; B)$ by the previous formula for its homogeneous components and take the sum. A simple norm estimates shows the convergence.
	\end{definition}
	
	\begin{proposition}[Second Transgression formula]
		\label{double transgression}
		\begin{equation}\label{2nd transgression}
		\tr\{f^{\pr}(\mc{R}_h) \cdot [\mE_h^{\pr},\theta]\} 
		= \del^X \tr\{f^{\pr}(\mc{R}_h) \cdot \theta\}
		\end{equation}	
	\end{proposition}

	\begin{proof} 
		For each $\mE_h$ along the vertical fiber $X$, we have 
		\begin{equation}\label{2nd transgression.1}
			d^X\tr \{ f^{\pr}(\mc{R}_h) \cdot \theta \}
			= \tr [\mE_h,f^{\pr}(\mc{R}_h) \cdot \theta]
		\end{equation}
		By the graded Jacobi identity, we can expand the right hand side of the above equation and compute
		\begin{align*}
			d^X\tr\{f^{\pr}(\mc{R}_h) \cdot \theta\}
			& = \tr\{[\mE_h,f^{\pr}(\mc{R}_h)] \cdot \theta\}
			+ \tr\{f^{\pr}(\mc{R}_h) \cdot [\mE_h,\theta]\} \\
			&=\tr\{f^{\pr}(\mc{R}_h;[\mE_h,\mc{R}_h]) \cdot  \theta\} + \tr\{f^{\pr}(\mc{R}_h) \cdot [\mE_h,\theta]\}
		\end{align*}
		By Bianchi identity, the first term is zero. Hence we have
		\begin{equation}\label{2nd transgression.2}
		d^X\tr\{f^{\pr}(\mc{R}_h) \cdot \theta\}
		= \tr\{f^{\pr}(\mc{R}_h) \cdot [\mE_h^{\pr},\theta]\}
		+ \tr\{f^{\pr}(\mc{R}_h) \cdot [\mE_h^{\dpr},\theta]\}
		\end{equation}
		Note as before, $[\mE_h^{\pr},\theta] \in \mc{G}^{-1}$, $[\mE_h^{\dpr},\theta] \in \mc{G}^{1}$ and $\tr \{f^{\pr}(\mc{R}_h) \cdot \theta\} \in \mc{G}^0$, comparing the $\mc{G}^{-1}$ component we have \eqref{2nd transgression.1}.
	\end{proof}
	
	Combining the first transgression formula (\ref{first transgression}) and second transgression formula (\ref{2nd transgression}), we established the double transgression formula for characteristic forms on cohesive modules.
	\begin{corollary}[Bott-Chern formula]
		The $\mc{M}$-directional differential of the characteristic form $f(E, \mb{E}^{\dpr}, h)$ at $h$ is given by:
		\begin{equation}\label{Bott Chern}
			d^{\mM}\tr f(\mc{R}_h) = \del^X\delbar^X\tr \{f^{\pr}(\mc{R}_h) \cdot \theta \}
		\end{equation}
	\end{corollary}

	This generalizes the classical formula obtained by Bott and Chern in \cite{Botta} for holomorphic Hermitian vector bundles. If we view the previous equality in Bott-Chern cohomology, the right hand side is zero and we established the invariance of characteristic classes under metric deformation. 
	
	\begin{corollary}
		The characteristic forms $f(E, \mb{E}^{\dpr}, h)$ in the Bott-Chern cohomology are independent of the choice of a Hermitian metric $h$.
	\end{corollary}
	
	In the last section, we will show the characteristic classes only dependent on the homotopy class of the cohesive module. 
	
	\subsection{Transgression formula for secondary classes}
	Our next goal is to study the differential forms $\tr(f^{\pr}(\mc{R}_h)\cdot \theta)$ that appear in the double transgression formula. As we will prove later,  they are the secondary characteristic classes and define holomorphic analog of the Chern-Simons forms. We will use them to study infinite determinant bundle and stability in a separate paper. For their applications in holomorphic vector bundles, see the classical work of Bismut, Gillet and Soul\'e in \cite{Bismut1988}\cite{Bismut1988b}\cite{Bismut1988c}.
	
	\begin{lemma}{\label{derivative of curvature}} 
		The curvature tensors $\mc{R}_h$ has its $\mM$-directional differential given by:
		\begin{equation}\label{derivative of curvature.1}
		d^{\mM} \mc{R}_h = [\mEh,[\mEhp,\theta]]
		\end{equation}
	\end{lemma}
	
	\begin{proof}
		By Bianchi identity for the universal cohesive module and use Lemma \ref{derivative of connection}, we have:
		\begin{equation}\label{derivative of curvature.2}
			0 = [\widetilde{\mb{E}}, \tR] = [{d^\mM} + \mEh,\Rh-[\mEhp,\theta]]
		\end{equation}
		Expand the terms and use Lemma \ref{derivative of connection} again, we have
		\begin{equation}\label{derivative of curvature.3}
			0 = d^{\mM}\Rh + [\mE_h,\mc{R}_h] - [\mE_h,[\mE_h^{\pr},\theta]] + (d^{\mM})^2(\mEh)
		\end{equation}
		By Bianchi identity for $\mb{E}_h$, the second term in (\ref{derivative of curvature.3}) vanishes and we get (\ref{derivative of curvature.1}).
	\end{proof}
	
	We will start to prove the main result of this section. Like the double transgression formula of Bott and Chern, the goal is to compute the $\mc{M}$-directional derivative of $\tr\{f^{\pr}(\Rh)\cdot \theta \}$ and show it's in the image of $\del^X$ and $\delbar^X$. We break the lengthy computation into several lemmas.
	
	\begin{lemma}\label{prep lemma 1 for third transgression}
		For any convergent power seris $g(T)$, we have
		\begin{align*}\numberthis\label{third transgression.7}
			d^X \tr \{g(\Rh; [\mEhp,\theta]) \cdot \theta \}
			= & \tr \{g(\Rh;[\mEh, [\mEhp, \theta]])\cdot \theta\} \\
			& + \tr \{g(\Rh;[\mEhp, \theta]) \cdot [\mEh, \theta] \}
		\end{align*}
	\end{lemma}
	
	\begin{proof}
		Without loss of generality, we assume $g(T) = T^n$. By our definition, $g\{\Rh;[\mEh,[\mEhp,\theta]]\}$ is a summation of the form:
		\begin{equation}\label{third transgression.4}
		g\{\Rh;[\mEh,[\mEhp,\theta]]\}
		= \sum_{i + j = n - 1} \Rh^{i} [\mEh,[\mEhp,\theta]] \Rh^j
		\end{equation} 
		If we consider the following differential, keeping in mind when passing $d^{X}$ over $\tr$, we act via $\mEh$ and follows the Leibniz rule, we have
		\begin{align*}\numberthis \label{third transgression.5}
			&d^X\sum_{i}\tr \{\Rh^{i}[\mEhp,\theta]\Rh^{n-i-1}  \theta\}
			= \sum_{i} d^X\tr \{\Rh^{i}[\mEhp,\theta]\Rh^{n-i-1}  \theta\} \\
			= & \sum_{i + j + k = n - 2} \tr \{\Rh^{i} [\mEh,\Rh]\Rh^{j} [\mEhp,\theta] \Rh^{k}\theta \}
			+ \sum_{i=1}^{n} \tr \{\Rh^{i-1} [\mEh,[\mEhp,\theta]] \Rh^{n-i} \theta\} \\
			+ & \sum_{i + j + k = n - 2}\tr \{\Rh^{i} [\mEhp,\theta] \Rh^{j} [\mEh,\Rh] \Rh^{k}  \theta \}
			+ \sum_{i=1}^{n} \tr \{\Rh^{i-1} [\mEhp,\theta] \Rh^{n-i}  [\mEh,\theta]\}
		\end{align*}
		The terms in the first and third summations are zero by Bianchi identity $[\mEh, \Rh] = 0$. Note that the second and last summations can be identified with $\tr\{g(\Rh;[\mEh, [\mEhp, \theta]])\cdot \theta \}$ and $\tr \{g(\Rh;[\mEhp, \theta]) \cdot [\mEh, \theta] \}$ respectively, the result follows.
	\end{proof}
	
	\begin{lemma}\label{prep lemma 2 for third transgression}
		The $\mM$-directional differential of $\tr \{ g(\Rh) \cdot \theta \}$ satisfies the equality:
		\begin{align*}\numberthis\label{third transgression.10}
			d^{\mM} \tr \{g(\Rh) \cdot \theta \} + \tr \{ g(\Rh) \cdot \theta^2\} 
			= & \delbar^X \tr \{g(\Rh;[\mEhp,\theta]) \cdot \theta \} \\
			& - \tr \{ g(\Rh;[\mEh,\theta]) \cdot [\mEdp,\theta] \}
		\end{align*}
		In addition, we have the following identities:
		\begin{equation}\label{third transgression.11}
			\tr \{g(\Rh; [\mEhp,\theta]) \cdot [\mEhp, \theta] \}=0
		\end{equation}
		\begin{equation}\label{third transgression.11.2}
			\tr \{g(\Rh;[\mEdp,\theta])\cdot[\mEdp,\theta] \}=0
		\end{equation}
	\end{lemma}
	
	\begin{proof}
		Without loss of generality, we assume $g(T) = T^n$ is a monomial. By Leibniz formula, we get	
		\begin{equation}\label{third transgression.2}
			d^{\mM}\tr \{g(\Rh) \cdot \theta \}
			= \tr \{g(\Rh;d^{\mM}\Rh) \cdot \theta \}
			+ \tr \{g(\Rh)\cdot {d^{\mM}}\theta \}
		\end{equation}
		We can substitute $d^{\mM}\Rh$ and $d^{\mM}\theta$ in the above equation by  Lemmas \ref{MC-equation} and Lemma \ref{derivative of curvature} respectively. Then we have
		\begin{equation}\label{third transgression.3}
			d^{\mM}\tr \{g(\Rh) \cdot \theta \}
			=\tr \{g(\Rh;[\mEh,[\mEhp,\theta]]) \cdot \theta \}
			-\tr \{g(\Rh) \cdot \theta^2 \}
		\end{equation}
		Lemma \ref{prep lemma 1 for third transgression} together with equation (\ref{third transgression.3}) shows that
		\begin{align*}\numberthis\label{third transgression.8}
			&d^{\mM}\tr \{g(\Rh) \cdot \theta \} + \tr \{ g(\Rh) \cdot \theta^2 \}\\
			= & d^X\tr \{g(\Rh;[\mEhp,\theta]) \cdot \theta \}
			- \tr \{g(\Rh;[\mEhp,\theta]) \cdot [\mEh,\theta] \}
		\end{align*}
		In equation \eqref{third transgression.8}, the left hand side is a $2$-form on $\mM$ with value in $\mc{G}^0$. For the right hand side of the equation, keeping in mind that $\mEhp$ increases exotic degree by $-1$ while $\mEdp$ increases it by $1$, we can decompose $\tr \{g(\Rh;[\mEhp,\theta]) \cdot [\mEh,\theta] \}$ as a sum of its $\mc{G}^{-2}$ and $\mc{G}^0$ components as:
		\begin{align}{\label{third transgression.9}}
			\tr \{g(\Rh;[\mEhp,\theta]) \cdot [\mEh,\theta] \}
			= & \tr \{g(\Rh;[\mEhp,\theta]) \cdot [\mEhp,\theta] \} \\
			& + \tr \{g(\Rh;[\mEhp,\theta]) \cdot [\mEdp,\theta] \} \nonumber
		\end{align}
		Comparing the $\mc{G}^{-2}$ components in equation \eqref{third transgression.8}, we get equation \eqref{third transgression.11}. Taking its adjoint, we get \eqref{third transgression.11.2}. Finally if we compare the $\mc{G}^0$ components in equation \eqref{third transgression.8}, we get
		\begin{align*}
		d^{\mM} \tr \{g(\Rh) \cdot \theta\} + \tr \{g(\Rh) \cdot \theta^2\} 
		= & \delbar^X \tr \{g (\Rh;[\mEhp,\theta]) \cdot \theta\} \\
		& - \tr \{g (\Rh;[\mEhp,\theta]) [\mEdp,\theta]\}
		\end{align*}
		Adding the zero term $-\tr \{g(\Rh;[\mEdp,\theta])\cdot[\mEdp,\theta]\}$ to the right hand side, we get equation \eqref{third transgression.10}.
	\end{proof}
	
	\begin{lemma}\label{prep lemma 3 for third transgression}
		The last term $\tr \{g (\Rh;[\mEh,\theta]) \cdot [\mEdp,\theta]\}$ in equation \eqref{third transgression.10} is given by:
		\begin{align*}\numberthis\label{third transgression.16}
			\tr \{g (\Rh;[\mEh,\theta]) \cdot [\mEdp,\theta] \} =
			& \del^X\tr \{ g(\Rh;[\mEdp,\theta])\cdot\theta \} \\
			& +\tr\{ g(\Rh;\theta)\cdot[\mEhp,[\mEdp,\theta]] \}
		\end{align*}
		And the term $\tr\{ g(\Rh;[\mEdp,\theta])\cdot\theta \}$ satisfies:
		\begin{equation}\label{third transgression.17}
			\delbar^X \tr \{ g(\Rh;[\mEdp,\theta]) \cdot \theta \}
			+ \tr \{g(\Rh;\theta) \cdot [\mEdp, [\mEdp, \theta]] \} = 0
		\end{equation} 
	\end{lemma}
	
	\begin{proof}
		We consider the following differential and compute it as in Lemma \ref{prep lemma 1 for third transgression}. Again we assume without loss of generality that $g(T) = T^n$.
		\begin{align*}\numberthis\label{third transgression.13}
			& d^X \sum_{i=1}^{n} \tr \{\Rh^{i-1}{\theta} \Rh^{n-i}[\mEdp,\theta] \} \\
			= & \sum_{i + j + k = n - 2 } \tr \{\Rh^{i} [\mEh,\Rh] \Rh^{j} {\theta} \Rh^{k} [\mEdp,\theta] \}
			+\sum_{i=1}^{n}\tr \{\Rh^{i-1}[\mEh,\theta]\Rh^{n-i}[\mEdp,\theta] \}\\	
			- & \sum_{i + j + k = n - 2} \tr \{\Rh^{i} {\theta} \Rh^{j} [\mEh,\Rh]\Rh^{k} [\mEdp,\theta] \}
			-\sum_{i=1}^{n} \tr \{\Rh^{i-1} \theta \Rh^{n-i} [\mEh,[\mEdp,\theta]] \}
		\end{align*}
		Again the terms in the first and third summations are zero by Bianchi identity. So we get  	 
		\begin{align*}\numberthis\label{third transgression.14}
			\tr \{g(\Rh;{[\mEh,\theta]})\cdot[\mEdp,\theta] \}  	    
			= & d^X\tr \{g(\Rh;{\theta})\cdot[\mEdp,\theta] \} \\
			&+ \tr\{g(\Rh;{\theta})\cdot[\mEh,[\mEdp,\theta]] \}
		\end{align*}
		Substitute \eqref{third transgression.14} into \eqref{third transgression.10} in Lemma \ref{prep lemma 2 for third transgression}, we have 	
		\begin{align*}\numberthis\label{third transgression.15}
		&d^{\mM}\tr \{g(\Rh)\cdot\theta \}+\tr \{ g(\Rh)\cdot\theta^2 \}\\
		=&\delbar^X \tr \{ g(\Rh; [\mEhp, \theta]) \cdot\theta \}
		- d^X \tr \{ g(\Rh; [\mEdp, \theta]) \cdot \theta \}\\
		-&\tr \{g(\Rh; \theta) \cdot [\mEh, [\mEdp, \theta]] \}
		\end{align*}
		Note again the left hand side is a $2$-form on $\mM$ with values in $\mc{G}^0$, we get the first equality by comparing $\mc{G}^0$ components and the second equality in by comparing $\mc{G}^{2}$ components.
	\end{proof}
	
	Combining the formulas we derived so far, we are ready to prove the following main theorem.
	\begin{theorem}\label{third transgression}
		Let $g(T)$ be a convergent power series in $T$, then the $\mM$-directional differential of $\tr \{g(\Rh)\cdot \theta \}$ is given by the following formula:  	    	
		\begin{equation}\label{third transgression.1}
		d^{\mM}\tr\{g(\Rh) \cdot \theta \}
		=\frac{1}{2}\delbar^X\tr\{g(\Rh;[\mEhp,\theta]) \cdot \theta \}
		-\frac{1}{2}\del^X\tr\{g(\Rh;[\mEdp,\theta]) \cdot \theta \}
		\end{equation}
	\end{theorem}

	\begin{proof}
		We write $\tr \{g(\Rh;\theta)[\mEh,[\mEdp,\theta]] \}$ as the sum of its $\mc{G}^0$ and $\mc{G}^{2}$ components  	
		\begin{equation}\label{third transgression.18}
		\tr\{g(\Rh;\theta)[\mEhp,[\mEdp,\theta]] \}
		+\tr\{g(\Rh;\theta)[\mEdp,[\mEdp,\theta]] \}
		\end{equation}
		then by Jacobi identity and flatness of $\mb{E}^{\dpr}$, we have 	
		\begin{equation}\label{third transgression.19}
		[\mEdp,[\mEdp,\theta]] = \frac{1}{2} [[\mEdp, \mEdp], \theta] = 0
		\end{equation}
		Similarly, we compute
		\begin{equation}\label{third transgression.20}
		[\mEhp,[\mEdp,\theta]]=[[\mEhp,\mEdp],\theta]-[\mEdp,[\mEhp,\theta]]
		\end{equation}
		By equation \eqref{third transgression.11}, we can add $0 = [\mEhp,[\mEhp,\theta]]$ to the above equation and note that $\mc{R}_h = [\mEhp, \mEdp]$,  we have
		\begin{equation}\label{third transgression.21}
		[\mEhp,[\mEdp,\theta]]= [\Rh,\theta]-[\mEh,[\mEhp,\theta]]
		\end{equation}
		By lemma \ref{derivative of curvature}, we have $d^{\mM}(\mc{R}_h)=[\mEh,[\mEhp,\theta]]$ and therefore we can rewrite the above formula as
		\begin{equation}
		[\mEh, [\mEdp, \theta]] = [\mEhp, [\mEdp, \theta]] 
		= [\Rh, \theta] - d^{\mM}\Rh
		\end{equation}
		so we have
		\begin{align*}\numberthis\label{third transgression.22}
		\tr\{g(\Rh;\theta)[\mEh,[\mEdp,\theta]] \}
		= \tr\{g(\Rh;\theta)[\Rh,\theta] \}
		- \tr\{g(\Rh;\theta)d^{\mM}\Rh \}
		\end{align*}
		By the property of $\tr$ and $\mc{R}_h$ has even total degree, we have
		\begin{align*}\numberthis\label{third transgression.23}
		\tr\{g(\Rh;\theta)d^{\mM}\Rh \}
		=&\sum_{i=1}^{n}\tr\{\Rh^{i-1}\theta\Rh^{n-i}d^{\mM}\Rh \}\\
		=&\sum_{i=1}^{n}\tr\{\theta\Rh^{n-i}d^{\mM}\Rh\Rh^{i-1} \}\\
		=&\tr\{\theta \cdot g(\Rh;d^{\mM}\Rh) \}
		\end{align*}
		Using this equation, we can rewrite the last term in equation \eqref{third transgression.16} in Lemma \ref{prep lemma 3 for third transgression} as:
		\begin{align*}
		&\tr\{g(\Rh;\theta)[\mEhp,[\mEdp, \theta]] \} \\
		=&\tr\{g(\Rh;\theta) \cdot [\Rh, \theta] \}
		-\tr\{g(\Rh;\theta) \cdot d^{\mM}\Rh  \} \\
		=&\tr\{g(\Rh;\theta) \cdot [\Rh, \theta]  \} +
		\tr \{g(\Rh;d^{\mM}\Rh) \cdot \theta \} \\ 
		\numberthis\label{third transgression.24}
		=&\tr \{ g(\Rh; \theta) \cdot [\Rh,\theta] \} + d^{\mM}\tr \{g(\Rh) \cdot \theta \}
		 +\tr\{g(\Rh)\theta^2 \}
		\end{align*}
		where the last equality follows from Maurer-Cartan equation for $\theta$.
		Finally we plug the equation \eqref{third transgression.24} into equation \eqref{third transgression.15}, after collecting terms, we get
		\begin{align*}\numberthis\label{third transgression.25}
		2d^{\mM} \tr \{g(\Rh) \cdot  \theta\} + 2\tr \{g(\Rh) \cdot \theta^2 \}
		=& \delbar^X \tr \{g(\Rh; [\mEhp, \theta]) \cdot \theta \} \\
		&- \del^X \tr \{g (\Rh; [\mEdp, \theta]) \cdot \theta \} \\
		&- \tr \{g(\Rh; \theta) \cdot [\Rh, \theta] \}
		\end{align*}
		We expand the last term $\tr\{g(\Rh; \theta)[\Rh, \theta]\}$ explicitly to get
		\begin{align*}
		\tr \{g(\Rh;\theta) \cdot [\Rh,\theta] \}
		&=\sum_{i=1}^{n} \tr \{\Rh^{i-1} \theta \Rh^{n-i} \Rh \theta \} 
		- \sum_{i=1}^{n} \tr \{\Rh^{i-1} \theta \Rh^{n-i} \theta \Rh \}\\
		&=\sum_{i=1}^{n} \tr \{\Rh^{i-1} \theta \Rh^{n-i+1} \theta \} 
		-\sum_{i=1}^{n} \tr \{\Rh^{i} \theta \Rh^{n-i} \theta \} \\
		&=\tr(\theta\Rh^n\theta)-\tr(\Rh^n\theta^2) = -2\tr \{\Rh^n\theta^2 \} \\ \numberthis\label{third transgression.26}
		&= -2\tr \{g(\mc{R}_h)\cdot \theta^2 \}
		\end{align*}
		Plug equation \eqref{third transgression.26} into \eqref{third transgression.25}, we have
		\begin{align*}\numberthis\label{third transgression.27}
		2 d^{\mM} \tr \{g(\Rh) \cdot \theta \}
		= \delbar^X\tr \{g(\Rh;[\mEhp,\theta]) \cdot \theta \} 
		-\del^X \tr \{g(\Rh;[\mEdp,\theta]) \cdot \theta \}
		\end{align*}
		Divide both sides by 2, we get the desired formula.
	\end{proof}
	
	We are now ready to define secondary Bott-Chern classes for cohesive modules with Hermitian structures. Recall $\mc{G}=\mc{G}^0$ is the space of $(p,p)$ forms, let $\mc{G}^{\pr}$ be the subspace of $\mc{G}$ defined by $\mr{Im}\del^X+\mr{Im}\delbar^X$. We will define the secondary classes as elements in ${\mc{G}^0}/{\mc{G}^{\pr}\cap \mc{G}^0}$.
	
	\begin{definition}
		Assume $k_1, k_2$ be two Hermitian metrics on a cohesive module $E$. For a convergent power series $f(T)$, we define the secondary Bott-Chern form $\tilde{f}(k_1, k_2)$ associated to $k_1, k_2$ as an element in ${\mc{G}^0}/{\mc{G}^{\pr}\cap \mc{G}^0}$ with a representatives $\tilde{f}(k_1, k_2; \gamma)$ in $\mc{G}^0$ given by:
		\begin{equation}
		\tilde{f}(k_1, k_2; \gamma) = \int_{\gamma}\tr\{f^{\pr}(\mc{R}_h)\cdot \theta \}d\gamma
		\end{equation}
		where $\gamma(t)$ is a curve on $\mc{M}$ that connects $k_1$ to $k_2$. The following proposition shows that this is well-defined.
	\end{definition}
	
	\begin{proposition}
		The equivalence class of $\tilde{f}(k_1, k_2; \gamma)$ in ${\mc{G}^0}/{\mc{G}^{\pr}\cap \mc{G}^0}$ is independent of $\gamma$.
	\end{proposition}
	
	\begin{proof}
		Let $\tau$ be a another path connecting $k_1$ to $k_2$, then by convexity of $\mc{M}$, the loop $\eta = \gamma-\tau$ is the boundary of a smooth 2-simplex $\sigma$ in $\mM$. By Stokes formula, we have 
		\begin{equation}
		\tilde{f}(k_1,k_2; \gamma)-\tilde{f}(k_1, k_2; \tau) 
		= \int_{\del{\sigma}} \tr \{f^{\pr}(\mc{R}_h) \cdot \theta \} d\eta
		= \int_{\sigma}d^{\mM}\tr \{f^{\pr}(\mc{R}_h) \cdot \theta \} d\sigma
		\end{equation}
		By theorem \ref{third transgression}, the integrand is an element in $\mc{G}^{\pr}$, the result follows.
	\end{proof}  
	
	\begin{corollary} 
		The primary characteristic form is related to the secondary form via the following equation:
		\begin{equation}
			f(E, h_2)-f(E, h_1) = -\delbar^X\del^X \tilde{f}(h_1, h_2;\gamma)
		\end{equation}
		for any path $\gamma$ that connects $h_1$ to $h_2$. 
	\end{corollary}
	
	\begin{proof}
		By the Bott-Chern formula  \eqref{Bott Chern}, we have
		\begin{align}
		f(E, h_2)-f(E, h_1)
		& = \int_{\gamma}d^{\mM}\tr f(\Rh)d\gamma\\
		& = - \int_{\gamma} \delbar^X \del^X \tr \{f^{\pr}(\Rh) \cdot \theta \}d\gamma 
		\end{align}
		the result follows from the definition of $\tilde{f}(h_1, h_2;\gamma)$.
	\end{proof}
	
	\section{Invariance of Bott-Chern classes under quasi-isomorphism}
	\subsection{Invariance under gauge transformation}
	In this section, we study the Bott-Chern forms under deformation of the cohesive structures. Unlike the situation of Hermitian metrics, the characteristic classes will depend on the cohesive structures. In fact, this is why we want to refine the characteristic classes to take value in Bott-Chern cohomology since their image in deRham cohomology only depend on the underlying topological complex vector bundle structure and therefore are independent of the cohesive structure.
	
	\begin{definition}
		For a $\mb{Z}$-graded Hermitian vector bundle $(E^{\bullet}, h)$, define $\mc{E}$ to be the space of $\delbar^X$-superconnections of total degree 1 and define $\mc{E}^{\dpr}$ to be subspace of $\mc{E}$ whose elements satisfy the flatness condition.
	\end{definition}
	
	\begin{definition}
		Recall the subspaces $\mc{G}^{\bullet}$ defined by exotic degrees in section 2. For each $k \in \mb{Z}$, we define $\mc{G}^{\dpr, k}$ to be the subspace of $\mc{G}^{k}$ whose elements are forms of type $(0, \bullet)$. Similarly we define $\mc{G}^{\pr, k}$ to be the subspace of $\mc{G}^k$ whose elements are forms of type $(\bullet, 0)$.  
	\end{definition}
	
	\begin{example}
		If $E^{\bullet}$ is concentrated in degree 0 and we set $\mc{A}$ to be the space of $\delbar^X$-connections and $\mc{A}^{\dpr}$ to be the subspace of $\delbar^X$-flat connections, then by Koszul-Malgrange theorem, $\mc{A}^{\dpr}$ is the space of holomorphic structures on $E$. If the underlying complex manifold $X$ is a Riemann surface, then every $\delbar^X$-connection is automatically flat and $\mc{A}^{\dpr} = \mc{A}$ is an affine space modeled on $\mc{A}^{0, 1}(X, \End E)$.
	\end{example}

	We assume $\mc{E}^{\dpr}$ is non-empty and it has a marked point $\mb{E}^{\dpr}$ such that $E = (E^{\bullet}, \mb{E}^{\dpr}, h)$ is a Hermitian cohesive module. The following lemma is a straightforward consequence of the explicit construction of Chern superconnections in Proposition \ref{Chern superconnection existence}.
	\begin{lemma}
		$\mc{E}$ is an affine space modeled on $\mc{G}^{\dpr, 1}$. $\mc{E}^{\dpr}$ is the subspace whose elements $\alpha^{\dpr} \in \mc{G}^{\dpr, 1}$ satisfy the Maurer-Cartan equation:
		\begin{equation}
		\mb{E}^{\dpr}(\alpha^{\dpr}) + \frac{1}{2}[\alpha^{\dpr}, \alpha^{\dpr}] = 0
		\end{equation}
		The $*$ operation interchanges $\mc{G}^{\dpr, \bullet}$ and $\mc{G}^{\pr, -\bullet}$. If $\mb{E}$ is the Chern superconnection of the Hermitian cohesive module $E$, then the space of Chern superconnection is in one-to-one correspondence with $\mc{E}^{\dpr}$. The correspondence is given by:
		\begin{equation}
		\mb{F} \xrightarrow{\delbar^X- \mr{ component}} \mb{F}^{\dpr}
		\end{equation}
		whose inverse is given by:
		\begin{equation}
		\mb{E} + \alpha^{\pr} + \alpha^{\dpr} \xleftarrow{\alpha^{\pr} = (\alpha^{\dpr})^*} \mb{E}^{\dpr} + \alpha^{\dpr}
		\end{equation}
	\end{lemma}
	
	\begin{definition}
		By the above lemma, the tangent space $T_{\mb{E}^{\dpr}}\mc{E}^{\dpr}$ of $\mc{E}^{\dpr}$ at a point $\mb{E}^{\dpr}$ can be identified with the space of solutions in $\mc{G}^{\dpr, 1}$ to the equation:
		\begin{equation}
		\mb{E}^{\dpr}(\alpha^{\dpr}) = 0
		\end{equation}
		In particular, the tangent space  is a subspace in $\mc{G}^{\dpr, 1}$. We define a one form $\delta^{\dpr}$ on $\mc{E}^{\dpr}$ with values in $\mc{G}^{\dpr, 1}$ by the formula:
		\begin{equation}
		\delta^{\dpr}(\alpha^{\dpr}) = \alpha^{\dpr}, \ \forall \alpha^{\dpr} \in T_{\mb{E}^{\dpr}}\mc{E}^{\dpr}
		\end{equation}  
		By duality, we also define the one form $\gamma^{\pr}$ on $\mc{E}^{\dpr}$ with values in $\mc{G}^{\pr, -1}$ by the formula:
		\begin{equation}\label{Chern superconnection correspondence with flat structure}
		\delta^{\pr}(\alpha^{\dpr}) = (\alpha^{\dpr})^* = \alpha^{\pr}
		\end{equation}
		Finally, we set $\delta = \delta^{\pr} + \delta^{\dpr}$ to be an one form with values in $\mc{G}^{\dpr, 1}\oplus \mc{G}^{\pr, -1}$.
	\end{definition}
	
	\begin{lemma}
		If $\mb{E}_t$ is a family of Chern superconnections such that $\mb{E}^{\dpr}|_{t = 0} = \mb{E}^{\dpr}$ and $\mb{E}^{\dpr}_t = \mb{E}^{\dpr} + A^{\dpr}_t$, then the deformation of the curvatures $\mc{R}_t$ is given by:
		\begin{equation}
			\frac{d}{dt}\mc{R}_t = [\mb{E}_t, \frac{d}{dt}\mb{E}_t]
		\end{equation}
		where $\mb{E}_t t$ is the Chern superconnection of $(E^{\bullet}, \mb{E}^{\dpr}_t, h)$. Using the one form $\delta$ defined above, we can rewrite this formulas as \begin{equation}
			d^{\mc{E}}\mc{R} = -[\mb{E}, \delta]
		\end{equation} 
		
	\end{lemma}
	\begin{proof}
		We write ${\alpha}_t = {\alpha}^{\pr}_t + {\alpha}^{\dpr}_t$ for the tangent vectors $\frac{d}{dt}\mE_t$ for short. For each $t$, ${\alpha}^{\dpr}_t$ satisfies:
		\begin{equation}
		[\mb{E}^{\dpr}_t, {\alpha}^{\dpr}_t] = 0
		\end{equation} and by duality, ${\alpha}^{\pr}_t$ satisfies
		\begin{equation}
		[\mb{E}^{\pr}_t, {\alpha}^{\pr}_t] = 0
		\end{equation} 
		By definition, $\mc{R}_t = [\mb{E}^{\dpr}_t, \mb{E}^{\pr}_t]$, so if we take its derivative, we have
		\begin{equation}
		\frac{d}{dt}\mc{R}_t = [{\alpha}^{\dpr}_t, \mb{E}^{\dpr}_t] + [{\alpha}^{\pr}_t, \mb{E}^{\pr}_t]
		\end{equation}
		Adding up the above equations, we get the desired equality.
	\end{proof}
	
	\begin{proposition}\label{first transgression over cohesive}
		Let $f(T)$ be a convergent power series in $T$, the $\mc{E}$-directional differential of the Bott-Chern characteristic form $f(E, \mb{E}^{\dpr}, h)$ at $\mb{E}^{\dpr} \in \mc{E}^{\dpr}$ is given by:
		\begin{equation}\label{1st transgression over moduli}
			d^{\mc{E}}f(\mc{R}) = - \tr f^{\pr}(\mc{R}\cdot [\mb{E}, \delta]) = -d^X \tr f^{\pr}(\mc{R} \cdot \delta)
		\end{equation} 
	\end{proposition}
	\begin{proof}
		The first equality is a direct consequence of previous lemma. The second equality follows from Bianchi identity $[\mb{E}, \mc{R}] = 0$ as before.
	\end{proof}
	
	\begin{corollary}
		We compare the $\mc{G}^{\bullet}$ components of both sides of the equality \eqref{1st transgression over moduli}. We get the following identities:
		\begin{align}
			\label{1st transgression eq1}
			& d^{\mc{E}}\tr f(\mc{R}) 
			  = -\del^X\tr \{ f^{\pr}(\mc{R})\cdot \delta^{\dpr} \} 
			  - \delbar^X \tr \{f^{\pr}(\mc{R})\cdot \delta^{\pr} \} \\
			\label{1st transgression eq2}
			& \del^X \tr \{ f^{\pr}(\mc{R})\cdot \delta^{\pr}\} = 0 \\
			\label{1st transgression eq3}
			& \delbar^X \tr \{f^{\pr}(\mc{R})\cdot \delta^{\dpr}\} = 0
		\end{align}
	\end{corollary}
	
	Equation \eqref{1st transgression eq1} is the first transgression formula over $\mc{E}^{\dpr}$. Motivated by the results we established in previous section, equation \eqref{1st transgression eq2} and \eqref{1st transgression eq3} are expected to admit a further transgression. However this is not the case in general for otherwise we would have proved $\tr f(\mc{R})$ is even independent of the cohesive structure. We will show the terms $\tr (f^{\pr}(\mc{R})\cdot \delta^{\pr})$ and $\tr (f^{\pr}(\mc{R})\cdot \delta^{\dpr})$ admit a further transgression formula when restricted to certain subspaces in $\mc{E}^{\dpr}$.
	
	\begin{definition}
		We define the generalized Dolbeault complex associated to a cohesive module $E$ by $(\mc{G}^{\dpr, \bullet}, \mb{E}^{\dpr})$. By the flatness condition on $\mb{E}^{\dpr}$, it defines a differential. The cohomology of the complex is defined to be the Dolbeault cohomology of the cohesive module $E$.
	\end{definition}
	
	\begin{remark}
		The Euler characteristic $\chi(E)$ of the Dolbeault cohomology group of a cohesive module $E$ is studied in a different paper \cite{Hua16b}. By Hodge theorey, $\chi(E)$ is given by the index of a Dirac-type operator $D^E$ defined by $\sqrt{2}(\mE^{\dpr} + (\mE^{\dpr})^*)$ where $(\mE^{\dpr})^*$ is the formal adjoint of $\mE^{\dpr}$ with respect to a Hermitian metric $h_E$ on $E$ and a Hermitian metric $h_X$ on $X$. It is proved that the index is computed by the classical Atiyah-Singer index formula:
		\begin{equation}
			\mr{Ind}(D^E) = (2\pi i)^{-\dim X}\int_X \mr{Todd}(X) \cdot \mr{ch}(E)
		\end{equation}
		where $\mr{Todd}(X)$ is the Todd genus and $\mr{ch}(E)$ is the Chern character form associated to the power series $\exp(-T)$.
	\end{remark}
	
	\begin{definition}
		Let $\mb{E}^{\dpr}_t$ be a one parameter family of cohesive structures. The tangent vectors $\alpha^{\dpr}_t = \dot{\mE}^{\dpr}_t$ satisfies $\mb{E}^{\dpr}_t({\alpha}^{\dpr}_t) = 0$ so they are pointwise closed with respect to the generalized Dolbeault operator $\mb{E}^{\dpr}_t$. We say the family is exact if there exist a smooth section $\gamma^{\dpr}_t$ with values in $\mc{G}^{\dpr, 0}$ such that 
		\begin{equation}
			[\mE^{\dpr}_t, \gamma^{\dpr}_t] = \delta^{\dpr}_t, \forall t
		\end{equation}
		That is to say, the tangent vectors $\alpha^{\dpr}_t$ defines zero cohomology class in the first Dolbeault cohomology groups and in addition, we can find a smooth lift of them. 
	\end{definition}
	
	\begin{example}
		Consider the group $\mr{GL}(E)$ whose elements are of the form $f = \sum_{k = 0} ^{\dim X} f_k$ where $f_k \in \mc{A}^{0, k}(X, \End^{-k}E)$ and $f_0$ is invertible. Since $\mc{A}^{\bullet > 0}$ is nilpotent and we required $f_0$ to be invertible, $\mr{GL}(E)$ forms a group. $\mr{GL}(E)$ acts on $\mc{E}$ and preserves the subspace $\mc{E}^{\dpr}$ via gauge transformation:
		\begin{equation}
			(\mE^{\dpr})^f = f^{-1} \circ \mE^{\dpr} \circ f = f^{-1}\circ [\mb{E}^{\dpr}, f] + \mb{E}^{\dpr}
		\end{equation}
		If $f_t$ is a one parameter family of gauge group elements in $\mr{GL}(E)$ such that $f_0 = \mr{Id}_E$, then for any choice of $\mb{E}^{\dpr} \in \mc{E}^{\dpr}$, we claim  the one parameter family of cohesive structures $\mE^{\dpr}_t = (\mb{E}^{\dpr})^{f_t}$ is exact.
		
		To see this, we simply take $\gamma^{\dpr}_t = f_t^{-1} \frac{d}{dt} f_t$ with values in $\mc{G}^{\dpr, 1}$ and we can compute
		\begin{align}
			&\frac{d}{dt}\mb{E}^{\dpr}_t 
			 = \frac{d}{dt}(f_t^{-1}\circ [\mb{E}^{\dpr}, f_t] + \mb{E}^{\dpr}) \\
			& = - \gamma^{\dpr}_t \circ f_t^{-1} \circ [\mE^{\dpr}, f_t] 
			+ f_t^{-1} \circ [\mE^{\dpr}, f_t \circ \gamma^{\dpr}_t] \\
			& = - \gamma^{\dpr}_t \circ \mb{E}^{\dpr}_t +  \gamma^{\dpr}_t \circ\mE^{\dpr} + \mE^{\dpr}_t \circ \gamma^{\dpr}_t -  \gamma^{\dpr}_t \circ \mE^{\dpr} \\
			& = [\mE^{\dpr}_t, \gamma^{\dpr}_t] 
		\end{align}
		This shows that the family $\mb{E}^{\dpr}_t$ obtained by applying a family of gauge transformations is exact.
	\end{example}
	
	\begin{remark}
		If $E$ is again just a holomorphic vector bundle concentrated in degree $0$, then $\mr{GL}(E)$ is just the group of invertible linear automorphisms. Two holomorphic vector bundle structure on $E$ are equivalent if and only if they differ by a gauge transformation.
	\end{remark}
	
	\begin{definition}
		If $\mc{S}$ is a submanifold of $\mc{E}^{\dpr}$ for which the restriction $\delta^{\dpr}$ is exact and admits a smooth lift $\gamma^{\dpr} \in \mc{G}^{\dpr, 0}$ such that $[\mb{E}^{\dpr}, \gamma^{\dpr}] = \delta^{\dpr}$ , we say $\mc{S}$ is exact.
	\end{definition}

	\begin{proposition}\label{second transgression over cohesive}
		If $\mc{S}$ is exact with a lift $\gamma^{\dpr}$ of $\delta^{\dpr}$ and we set $\gamma^{\pr} =  (\gamma^{\dpr})^*$, then we have the following identities:
		\begin{align}
		\label{1st transgression eq4}
		& \delbar^X \tr \{ f^{\pr}(\mc{R}) \cdot \gamma^{\dpr} \} = \tr \{ f^{\pr}(\mc{R}) \cdot \delta^{\dpr} \} \\
		\label{1st transgression eq5}
		& \del^X \tr \{f^{\pr}(\mc{R}) \cdot \gamma^{\pr} \} = - \tr \{ f^{\pr}(\mc{R}) \cdot \delta^{\pr} \}
		\end{align}
	\end{proposition}
	\begin{proof}
		By Bianchi identity, we have
		\begin{equation}
			d^X \tr \{f^{\pr}(\mc{R}) \cdot \gamma^{\dpr} \} = \tr \{f^{\pr}(\mc{R}) \cdot [\mE, \gamma^{\dpr}] \}
		\end{equation}
		We compare the $\mc{G}^2$ component of the equation and we get the first equation \eqref{1st transgression eq4}. If we replace $\gamma^{\dpr}$ by $\gamma^{\pr} = (\gamma^{\dpr})^*$, it satisfies 
		\begin{equation}
			[\mE^{\pr}, \gamma^{\pr}] = - \delta^{\pr}
		\end{equation}	
		by taking adjoints of $[\mb{E}^{\dpr}, \gamma^{\dpr}] = \delta^{\dpr}$. The same argument as above shows the second equation \eqref{1st transgression eq5}.
	\end{proof}
	
	Combining Proposition \ref{first transgression over cohesive} and Proposition \ref{second transgression over cohesive}, we have the following double transgression formula for Bott-Chern forms over an exact submanifold.
	\begin{proposition}\label{double transgression cohesive}
		With the same assumption as above, if we set $\gamma = \gamma^{\pr} + \gamma^{\dpr} \in \mc{G}^{0}$, the $\mc{S}$-directional differential of Bott-Chern forms is given by:
		\begin{equation}\label{double transgression cohesive eq}
		d^{\mc{S}}\tr f(\mc{R}) = -\del^X\delbar^X\tr \{f^{\pr}(\mc{R}) \cdot \gamma \}
		\end{equation}
	\end{proposition}
	
	As a corollary, the Bott-Chern characteristic classes remains invariant over $S$. We shall now apply this proposition to prove the invariance of Bott-Chern cohomology classes under homotopy equivalences. 
	
	For two cohesive modules $(E^{\bullet},\mE^{\dpr})$ and $(F^{\bullet},\mb{F}^{\dpr})$ , we have the following criteria for homotopy equivalence. For its proof, we refer to \cite{Block2010}.
	\begin{proposition}
		$(E^{\bullet},\mE^{\dpr})$ and $(F^{\bullet},\mb{F}^{\dpr})$ are equivalent in the homotopy category $\mr{Ho}(\mc{P}_{\mc{A}})$ if and only if there is a degree zero closed morphism $\phi \in \mc{P}^0_{\mc{A}}$ which is a quasi-isomorphism between the complexes $(E^{\bullet},\mE^{\dpr}_0)$ and $(F^{\bullet},\mb{F}^{\dpr}_0)$.
	\end{proposition}
	
	Using this criteria, the invariance of Bott-Chern classes will be proved in two steps. First we show that if the underlying complex $(E^{\bullet},\mE^{\dpr}_0)$ is acyclic, the characteristic classes are zero. Next we show that the Bott-Chern characteristic classes are additive with respect to short exact sequences of mapping cones and reduce to the acyclic case.
	
	\subsection{Transgression formula for acyclic complex}
	We start by define a rescaling operation on superconnections. If $t\in \mb{R}^{+}$ is a positive real constant and $\mE$ is a superconnection, define 
	\begin{equation}
	\mE_t = \sum_k t^{1-k}\mE_k
	\end{equation}
	It's easy to verify that $\mE^{\dpr}_t = (\mE^{\dpr})_t$ and $\mb{E}^{\dpr}_t$ are flat. So $\mE^{\dpr}_t$ forms a smooth family of cohesive structures.
	
	\begin{definition}
		We define the grading operator $N_E$ for the cohesive module $E$ as an element in $\mc{A}^{0}(X,\End^0(E^{\bullet}))$ such that it acts by:
		\begin{equation}
		N_E (A\otimes \omega)= \norm{A} \cdot A \otimes \omega
		\end{equation}
	\end{definition}
	
	\begin{lemma}\label{degree commutator}
		With $N_E$ so defined, we choose $\gamma^{\dpr}_t=\frac{1}{t}N_E \in \mc{G}^0$ and we have
		\begin{equation}\label{degree commutator eq:1}
			[\mE^{\dpr}_t,\gamma^{\dpr}_t] = \frac{1}{t}\sum_k (1-k)(\mE^{\dpr}_{k,t}) = \frac{d}{dt}\mb{E}^{\dpr}_t
		\end{equation}
	\end{lemma}
	\begin{proof} 
		\begin{align}
		& \frac{d}{dt}\mb{E}^{\dpr}_t
		= \sum_k (1 - k) t^{-k} \mb{E}^{\dpr}_k 
		=\frac{1}{t}\sum_k (1 - k) (\mb{E}^{\dpr}_{k})_{t}
		\end{align}
		On the other hand, we may take $s \in E^{d}$ so we have
		\begin{align}
		& \mb{E}^{\dpr}_t \circ N_E (s) 
		= d \sum_k t^{1-k} \mb{E}^{\dpr}_{k}(s)
		\end{align}
		and note that $\mb{E}^{\dpr}_k(s) \in \mc{A}^{0, k}(X, E^{d + 1 - k})$, so we have
		\begin{align}
		& N_E \circ \mb{E}^{\dpr}_{k, t}(s)
		= N_E \sum_k t^{1-k} \mb{E}^{\dpr}_k(s)
		= \sum_k t^{1-k} (d + 1 - k) \mb{E}^{\dpr}_k(s)
		\end{align}
		Taking the difference of the above two equalities and divide both sides by $t$, we get \eqref{degree commutator eq:1}.
	\end{proof}
	
	Since $N_E$ is clearly self adjoint, we have $\gamma^{\pr} = \gamma^{\dpr} = \frac{1}{t} N_E$ and we have the following corollary. 
	\begin{corollary} For the family of rescaled cohesive structures, we have
		\begin{equation}
		\frac{d}{dt}\tr(e^{-\mc{R}_t})
		= -\frac{2}{t}\del^X\delbar^X\tr(e^{-\mc{R}_t} \cdot N_E)
		\end{equation} 
	\end{corollary}
	
	Finally, by the same arguments in \cite{Bismut1988}, if $(E^{\bullet},\mb{E}^{\dpr}_0)$ is acyclic, then the degree zero component $\Delta_E$ of its curvature 
	$$\mc{R}_t^{0,0} = t^2 (\mb{E}^{\dpr}_0 + \mb{E}^{\pr}_0) ^ 2 
	= t^2 \Delta_{E}$$ 
	is a strictly positive element and the characteristic form of the Chern character
	$$\tr\exp(-R_t) =  \tr \exp(-t^2 \Delta_E + \mathcal{O}(t))$$
	decays exponentially fast uniformly on $X$ when $t$ approaches $\infty$. Applying this result and let $t \to \infty$, we transgressed the characteristic forms to zero as desired.
	
	\begin{theorem}\label{acyclic class}
		If $(E^{\bullet}, \mb{E}^{\dpr})$ is a cohesive module such that $(E^{\bullet}, \mb{E}^{\dpr}_0)$ is an acyclic complex, then the integral 
		\begin{equation}
		\mc{I}_E = \int_{1}^{\infty}\tr \{\exp({-\mc{R}_t}) \cdot N_E \}\frac{dt}{t}
		\end{equation}
		is finite and
		\begin{equation}
		\tr\exp({-\mc{R}})=\del^X\delbar^X \mc{I}_E
		\end{equation}
	\end{theorem}

	
	\subsection{Additive properties of Bott-Chern classes}
	We begin by defining the mapping cone of a closed morphism. Let $\phi \in \mc{P}^0_{\mc{A}}(E,F)$ be a closed degree $0$ morphism between two cohesive modules $(E^{\bullet}, \mE^{\dpr})$ and $(F^{\bullet}, \mb{F}^{\dpr})$.
	\begin{definition}
		The mapping cone $\textrm{Cone}(\phi)$ is the cohesive module $(\mr{Cone}(\phi),\mb{C}^{\dpr}_{\phi})$ whose underlying complex vector bundle is defined by:
		$$\mr{Cone}(\phi)^{\bullet}=F^{\bullet}\oplus E^{\bullet+1}$$
		and with respect to this decomposition, the cohesive structure $\mb{C}^{\dpr}_{\phi}$ is given by:
		$$\mb{C}^{\dpr}_{\phi}=
		\begin{pmatrix}
		& \mb{F}^{\dpr}  & \phi \\
		& 0              & -\mE^{\dpr}
		\end{pmatrix}$$
	\end{definition}
	
	Consider the one parameter family of morphisms $\phi_t = t\phi$ with $t \in [0, 1]$ such that $\phi_0 = 0$ and $\phi_1 = \phi$.  The corresponding mapping cones have the same underlying bundle $\mr{Cone}(\phi)^{\bullet}$ and cohesive structures $\mb{C}^{\dpr}_{t}$. It's a simple calculation that 
	\begin{equation}
	\frac{d}{dt}\mb{C}^{\dpr}_{\phi_t}=
	\begin{pmatrix}
	& 0 & \phi \\
	& 0 & 0
	\end{pmatrix}
	=\alpha^{\dpr}
	\end{equation}
	is constant. If we set for $t > 0$, $\gamma^{\dpr}_t \in \mc{G}^{\dpr, 0}$ by the formula:
	$$\gamma^{\dpr}_t=
	\begin{pmatrix}
	& 0 & 0 \\
	& 0 & \frac{1}{t}\cdot\mr{Id}_F
	\end{pmatrix}
	=\frac{1}{t}\gamma^{\dpr}$$
	then it satisfies that
	\begin{equation}
	[\mb{C}^{\dpr}_{t}, \gamma^{\dpr}_t]=\frac{d}{dt}\mb{C}^{\dpr}_{t}
	\end{equation}  
	So we can apply the previous Proposition \ref{double transgression cohesive} for all nonzero value of $t$. Note again in this case $\gamma^{\pr} = \gamma^{\dpr}$, we have the following equality of characteristic forms. 
	\begin{equation}
		f(\mr{Cone}(\phi),\mc{C}^{\dpr}_1)-f(\mr{Cone}(\phi),\mc{C}^{\dpr}_{s})
		=-\del^X\delbar^X \{\int_{s}^{1} 2\tr f^{\pr} (\mc{R}_t) \cdot \gamma^{\dpr}\frac{dt}{t} \}
	\end{equation}

	It's clear that we can't let $t\to 0$ in the above formula due to the singularity. Instead, motivated by the construction in \cite{Botta}, we will modify the integrand to remove the singularity. 
	
	To do this, we calculate $\mc{R}_t$ explicitly as
	\begin{equation}
	\mc{R}_t=
	\begin{pmatrix}
	&\mc{R}_F+t^2\phi\phi^* & t(\mb{F}^{\pr}\phi-\phi\mb{E}^{\pr}) \\
	&t(\phi^*\mb{F}^{\dpr}-\mb{E}^{\dpr}\phi^*) &\mc{R}_E+t^2\phi^*\phi
	\end{pmatrix}
	=\mc{R}_0+tA_t
	\end{equation}
	where $\mc{R}_0$ is the curvature computed by $\mb{C}^{\dpr}_{0}$ computed using the direct sum Hermitian form
	\begin{equation}
	\mc{R}_0=
	\begin{pmatrix}
	&\mc{R}_F & 0 \\
	&0 &\mc{R}_E
	\end{pmatrix}
	\end{equation}
	and $A_t$ is the reminder term
	\begin{equation}
	A_t=
	\begin{pmatrix}
	&t\phi\phi^* & \mb{F}^{\pr}\phi-\phi\mb{E}^{\pr} \\
	&\phi^*\mb{F}^{\dpr}-\mb{E}^{\dpr}\phi^* &t\phi^*\phi
	\end{pmatrix}
	\end{equation}
	If we evaluate a convergent power series $g(T)$ on $\mc{R}_t$, we have
	\begin{equation}\label{curvature expansion}
	g(\mc{R}_t)=g(\mc{R}_0+tA_t)=g(\mc{R}_0)+t R_{g}(t,\mc{R}_0,A_t)
	\end{equation}
	for some reminder term $R_g$ that is a power series in $t$ with coefficients that are polynomials in $\mc{R}_0$ and $A_t$. 
	
	We have already shown that $\tr f^{\pr}(\mc{R}_0)$ is $d^X$-closed by Proposition \ref{characteristic forms are closed}, therefore it is also $\del^X\delbar^X$-closed. So we can subtract it from the above equation and derive the following equation:
	\begin{equation}
		f(\mr{Cone}(\phi),\mc{C}^{\dpr}_1)-f(\mr{Cone}(\phi),\mc{C}^{\dpr}_{s})
		=-2\del^X\delbar^X\int_{s}^{1} \tr \{f^{\pr}(\mc{R}_t)\gamma^{\dpr} - f^{\pr} (\mc{R}_0)\gamma^{\dpr} \}\frac{dt}{t}
	\end{equation}
	Now we substitute equation \eqref{curvature expansion} with $g = f^{\pr}$ into the above equality, the singular term $\frac{1}{t}$ cancels out and we get
	\begin{equation}
	f(\mr{Cone}(\phi),\mc{C}^{\dpr}_1)-f(\mr{Cone}(\phi),\mc{C}^{\dpr}_{s})
	=-2\del^X\delbar^X\int_{s}^{1} \tr R_{f^{\pr}}(t,A_t,\mc{R}_0)dt
	\end{equation}
	Now the we can let $s \to 0$ since the integrand is bounded and we established the following proposition.
	
	\begin{proposition}
		The characteristic form $f(\mr{Cone}(\phi),\mc{C}^{\dpr}_1)$ coincide with $f(\mr{Cone}(\phi),\mc{C}^{\dpr}_0)$ in Bott Chern cohomology.
	\end{proposition}
	
	Under the direct sum Hermitian metric $h_E \oplus h_F$ on the mapping cone $\mr{Cone}^{\bullet}$, we have the following simple equality of differential forms:
	\begin{equation}\label{additive property}
	f(\mr{Cone}(\phi)^{\bullet},\mb{C}^{\dpr}_0,h_{\phi})=f(F^{\bullet},\mb{F}^{\dpr},h_F)-f(E^{\bullet},\mb{E}^{\dpr},h_E)
	\end{equation}
	where the minus sign comes from the shift in degree of $E^{\bullet}$ in the mapping cone. Together with the previous corollary, we established the following proposition.
	
	\begin{proposition}
		The Bott Chern cohomology classes of an exact triangle 
		$$0 \rightarrow (E^{\bullet}, \mb{E}^{\dpr}) \xrightarrow{\phi} 
		(F^{\bullet}, \mb{F}^{\dpr}) \xrightarrow{i_F} (\mr{Cone}^{\bullet}(\phi), \mb{C}_{\phi}) $$
		is additive in the sense
		\begin{equation}
		f(E^{\bullet}, \mb{E}^{\dpr}) - f(F^{\bullet}, \mb{F}^{\dpr}) 
		+ f(\mr{Cone}^{\bullet}(\phi), \mb{C}_{\phi}) = 0 
		\end{equation}
		in Bott-Chern cohomology
	\end{proposition}
	
	If $\phi$ is a homotopy equivalence, it was shown in \cite{Block2010} that this is equivalent to $\phi_0: (E^{\bullet},\mE^{\dpr}_0) \rightarrow (F^{\bullet},\mb{F}^{\dpr}_0)$ being a quasi-isomorphism. Therefore the mapping cone $\mr{Cone}(\phi)$ is acyclic so proposition \ref{acyclic class} shows that the left side of equation \eqref{additive property} is zero in Bott-Chern cohomology. Combining these, we established our main result below.
	
	\begin{theorem}
		If two cohesive modules $(E^{\bullet},\mE^{\dpr})$ and $(F^{\bullet},\mb{F}^{\dpr})$ are homotopy equivalent, then they have the same Bott-Chern cohomology classes.
	\end{theorem}
	
	As a corollary, we can define characteristic classes for an object $\ms{S}$ in $D^b_{\mr{Coh}}(X)$ as follows. We choose a cohesive module representative $(E^{\bullet}, \mb{E}^{\dpr})$ and equip it with some Hermitian structure $h_E$. Then for any convergent power series $f(T)$, we may define the associated Bott-Chern characteristic class $f(\ms{S})$ of $\ms{S}$ as the class of $f(E, \mb{E}^{\dpr}, h_E)$. In particular, this extends Bott-Chern cohomology classes to coherent sheaves.
	
	\bibliographystyle{numeric}
	\bibliography{library}

\begin{thebibliography}{BGS88b}

\bibitem[AF61]{Atiyah1961}
M.F. Atiyah and F.Hirzebruch.
\newblock {Analytic Cycles on Complex Manifolds}.
\newblock {\em Topology}, 1:25--45, 1961.

\bibitem[BC65]{Botta}
Raoul Bott and S~S Chern.
\newblock {Hermitian vector bunbles and the equidistribution of the zeros of
  their holomorphic sections.pdf}.
\newblock {\em Acta Mathematica}, 114(1):71--112, 1965.

\bibitem[BGS88a]{Bismut1988}
J.~M. Bismut, H.~Gillet, and C.~Soul{\'{e}}.
\newblock {Analytic torsion and holomorphic determinant bundles I. Bott-Chern
  forms and analytic torsion}.
\newblock {\em Communications in Mathematical Physics}, 115(1):49--78, 1988.

\bibitem[BGS88b]{Bismut1988b}
J.~M. Bismut, H.~Gillet, and C.~Soul{\'{e}}.
\newblock {Analytic torsion and holomorphic determinant line bundles II. Direct
  image and Bott-Chern forms.}
\newblock {\em Communications in Mathematical Physics}, 115(1):79--126, 1988.

\bibitem[BGS88c]{Bismut1988c}
J.~M. Bismut, H.~Gillet, and C.~Soul{\'{e}}.
\newblock {Analytic torsion and holomorphic determinant line bundles III.
  Quillen metrics on holomorphic determinants.}
\newblock {\em Communications in Mathematical Physics}, 115(1):301--351, 1988.

\bibitem[Blo06]{Block2006}
Jonathan Block.
\newblock {Duality and equivalence of module categories in noncommutative
  geometry II: Mukai Duality for Holomorphic Noncommutative tori}, 2006.

\bibitem[Blo10]{Block2010}
Jonathan Block.
\newblock {Duality and equivalence of module categories in noncommutative
  geometry I}.
\newblock {\em Centre de Recherches Mathematiques CRM Proceedings and Lecture
  Notes}, 50:1--30, sep 2010.

\bibitem[Don87]{Donaldson1987}
SK~Donaldson.
\newblock {Infinite determinants, stable bundles and curvature}.
\newblock {\em Duke Mathematical Journal}, 54(1):231--247, 1987.

\bibitem[Qia16]{Hua16b}
Hua Qiang.
\newblock {Index theorems for generalized Dolbeault-Dirac operator on coherent
  sheaves}, 2016.

\bibitem[Qui85]{Quillen1985}
Daniel Quillen.
\newblock {Superconnections and the Chern character}.
\newblock {\em Topology}, 24(1):89--95, 1985.

\end{thebibliography}
	
\end{document}